\documentclass[12pt]{amsart}

\usepackage[english]{babel}
\usepackage[utf8]{inputenc}
\usepackage[T1]{fontenc}

\usepackage[margin=1.5in]{geometry}

\usepackage[colorinlistoftodos]{todonotes}
\usepackage[colorlinks=true, allcolors=blue]{hyperref}

\usepackage[alphabetic]{amsrefs}

\usepackage[makeroom]{cancel}
\usepackage{setspace}
 \usepackage{changepage}
\usepackage{amsfonts} 
\usepackage{amssymb}
\usepackage{amsthm}
\usepackage{indentfirst}
\usepackage{enumerate}
 \usepackage{enumitem}   
\usepackage{amsmath}
\usepackage{graphicx}

\usepackage{mathabx}
\graphicspath{ {./images/} }

\usepackage{enumitem}
\usepackage[most]{tcolorbox}

\usepackage{tikz-cd}

\definecolor{myGreen}{RGB}{10 100 10}

\newcommand{\bC}{\mathbb{C}}

\newcommand{\bR}{\mathbb{R}}

\newcommand{\bF}{\mathbb{F}}
\newcommand{\bN}{\mathbb{N}}

\newcommand{\bA}{\mathbb{A}}

\newcommand{\cB}{\mathcal{B}}
\newcommand{\cC}{\mathcal{C}}

\newcommand{\cS}{\mathcal{S}}

\newcommand*{\tran}{\mathsf{T}}

\newtheorem{theorem}{Theorem}[section]
\newtheorem{lemma}[theorem]{Lemma}
\newtheorem{prop}[theorem]{Proposition}

\newtheorem{cor}[theorem]{Corollary}
\newtheorem{thm}[theorem]{Theorem}

\newtheorem*{cor*}{Corollary}
\newtheorem*{thm*}{Theorem}
\newtheorem*{lem*}{Lemma}

\newtheorem*{prop*}{Proposition}

\theoremstyle{definition}

\newtheorem*{defn*}{Definition}

\theoremstyle{remark}

\title[Convex sets and bases--classical and nc]{Compact convex sets and bases--classical and noncommutative}

\author[David Blecher]{David P. Blecher}
	
\address{Department of Mathematics, University of Houston, Houston, TX 77204-3008.}
	
\email{dpbleche@central.uh.edu}

\author[C. H. Pretorius]{Christiaan H. Pretorius}
\address{Department of Mathematics, University of Houston, Houston, TX 77204-3008.}
	
\email{chpretor@uh.edu}

 
\date{Revision of 7/7/2026}

\begin{document}

    \maketitle 

    \begin{abstract}   Matrix and noncommutative convexity constitute an important area of modern noncommutative analysis and have found significant applications in mathematical physics.   
In the first part of our paper we give an abstract characterization of  matrix convex sets, and   compact matrix convex sets.
      Our approach is in some part via a  universal   Banach space (resp.\ operator space)
    $X_K$  of an abstract compact convex set  (resp.\ matrix convex set) $K$.  This turns out to be a concrete construction
    of the base norm space (resp.\ nc base norm space) with base $K$, together with a natural TVS topology.   Noncommutative (nc for short) base norm spaces, recently developed by the first author and Hay,  are an important class of operator spaces which include 
    duals and preduals of unital $C^*$-algebras and von Neumann algebras, and operator systems, where the `base' is
    exactly the noncommutative convex set of (matrix) states on these.  In the later parts of the paper we give many applications, mostly  to base norm spaces (classical and noncommutative).     We also refine some of our recent results concerning regularity of convex  sets  
     (classical and noncommutative).  We give several interesting characterizations of base norm spaces   (classical and noncommutative).   Any such characterization will correspond by duality  to a new characterization of 
operator systems, or in the classical case, of function systems. For example, (complex) nc dual base norm spaces  are the matrix ordered  LCTVS's $V$
such that $V$ (at level 1) has a linear base  which is compact.     
       \end{abstract}

\section{Introduction} Matrix and noncommutative convexity constitute an important area of modern noncommutative analysis and have found significant applications in mathematical physics.  See, for example, the discussion and  references cited in the introductions of \cite{BMcI,BH,DK}.  Indeed  there are many 
 compelling expositions of these connections in recent literature, such as the work of the `Helton school' (we mention just \cite{EPS} since it is a survey), or as another example 
 forthcoming work of Kennedy and Skoufranis \cite{KS}.  
 In the first part of our paper, we provide an abstract characterization of  matrix convex sets, as well as of  compact matrix convex sets; a brief description of this characterization is given three paragraphs below.
 This of course in some sense gives a new characterization of operator systems in view of the well known duality between the latter and matrix convex sets \cite{WW,DK}.
    Before this we give a quick review of the classical real case of this, that is abstract characterizations of convex sets and compact convex sets, 
    and we will extend this to the complex case.    Our approach is via a  universal   Banach space (resp.\ operator space) $X_K$ of an abstract compact convex set  (resp.\ matrix convex set) $K$.  This turns out to be a concrete construction
    of the (unique) {\em base norm space} (resp.\ {\em nc base norm space}) with base $K$, together with a natural TVS topology.   Nc  base norm spaces are an important class of operator spaces which include 
    duals and preduals of unital $C^*$-algebras and von Neumann algebras, and operator systems, where the `base' is
    exactly the noncommutative convex set of (matrix) states on these.  In a very recent  paper \cite{BH} the first author and Hay  generalized the theory of base norm spaces to the complex case, and  further to the noncommutative setting relevant to `quantum convexity'.   
    In that paper we spell out the importance of this, e.g.\ the connection to GPT's (Generalized Probability Theories) in mathematical physics (see e.g.\ \cite{ALetal,Lami,DL}).   Or, as we said there, {\em Quantum Channels} or {\em CPTP} maps
    ({\em completely positive trace preserving}), are just our {\em base morphisms}, that is maps between matrices (viewed as nc base spaces in our language) preserving the base. The adjoint (dual) of such maps are the {\em UCP maps}.
     Indeed mathematical physicists usually prefer the base formulation, as any operator algebraist not familiar with physics knows who has tried to read a quantum physics or QIT article and found themselves having to `translate'  the predual or base formulations into statements at the algebra level.

    In the later parts of our paper we give many applications, mostly  to base norm spaces (classical and noncommutative).   For example we show how the weak* topology on the dual $M^*$ of a von Neumann algebra can be constructed 
     from the weak* topology on $M^*_+$, or from the weak* topology on the state space.   We also refine the ‘main regularity results' of \cite{Breg}, concerning regularity of convex  sets 
     (classical and noncommutative).  For example if $V$ is an LCTVS  and if $K$ is a compact convex set 
which spans $V$ and lies in a hyperplane in $V$ not passing through 0 
then $V$ is canonically isomorphic to $A(K)^*$, the dual of the space of continuous affine functions on $K$, via a continuous  isomorphism  which
is a homeomorphism on bounded sets.
Similarly in the nc case: If $V$ is a complex  $*$-LCTVS, and if $K$ is a selfadjoint compact matrix convex set 
such that $K_1$  spans $V_{\rm sa}$ and lies in a hyperplane not passing through 0, 
then $\bA(K,\bC)^* \cong V$ via a continuous   selfadjoint  isomorphism  which
is a homeomorphism on bounded sets. 
We also give several interesting characterizations of  dual operator systems, and base norm spaces   (classical and noncommutative).  
Again, any such characterization will correspond by the duality results in \cite{BH} to a new characterization of 
operator systems, or in the classical case, function systems.  For example, (complex) nc dual base norm spaces `are' the matrix ordered  $*$-LCTVS's $V$ with  closed  matrix cones such that (at level 1) $V_+$  is locally compact and spans $V$.    Or, 
such that $V$ (at level 1) has a linear base  which is $\tau$-compact (see Theorem \ref{breg25n}).   Here $\tau$ need not be the expected weak* topology,
but it may be switched with it, and in any case is the same on `bounded sets'.   (The majority of our characterizations are not however stated in terms of matrix ordered  spaces.)

Most of our paper is developed in the setting of Wittstock's  complex matrix convexity as in \cite{WW}, for example.
 Since our paper is already sizable, we will defer to a sequel  the real  matrix convexity case, which is a bit more technical, as well as the case relevant to Davidson and Kennedy's nc convexity \cite{DK}.    We also give a development of the `classical case', and of the complex variant of this, all of which is used later in the noncommutative  case. 
 
 The answer to  the question ``What is a(n abstract)  compact convex set?'' is: it is  an abstract convex set $K$  which is compact and topologically convex (this is defined at the start of Subsection \ref{wiacc}).  Moreover 
such a set `is' a compact convex set in an LCTVS if and only if it is locally convex  (that is, $K$ has a basis of convex neighborhoods).  See Theorem \ref{lmjot}.  
In Section \ref{wiamcsccs} we will answer the question ``What is a  
noncommutative compact convex set?'' 
in the complex case, for example as the matrix sets $K$ in the following result:
 
 \begin{thm} \label{defmck}  Let $K$ be a `sequence' of sets $(K_n)_{n \in \bN}$ (resp.\ $(K_n)_{n \leq \kappa}$, for a fixed cardinal $\kappa$) such that $K_1$ is an 
abstract compact, topologically convex, (resp.\ and locally convex) set and we have 
actions $\bC^n \times K_n \to K_1$ for all $n \in \bN$ (resp.\ $n \leq \kappa$), 
written as $\xi^* v \xi$ for $v \in K_n, \xi \in \bC^n, \| \xi \| = 1$, 
satisfying three conditions
 {\rm {\bf (M1)$^{\prime}$}},  {\rm {\bf (M2)$^{\prime}$}}, and 
 {\rm {\bf (M3)$^{\prime}$}}.
We also assume that  $K_n$ has a compact topology for $n \in \bN$ (resp.\ $n \leq \kappa$) with 

\medskip

\begin{adjustwidth}{1em}{1em} {\bf (M4)} $\; \;   \; \xi^* v \xi$ is continuous in $v \in K_n$ for fixed $\xi \in \bC^n, \| \xi \| = 1$. \end{adjustwidth}

\medskip 

If these all hold then $K$ is abstractly matrix (resp.\ nc) affinely homeomorphic to  a (complex) matrix (resp.\ nc) compact convex set in 
a TVS (resp.\  dual operator space).   Indeed if $K_1$ is locally convex then the
 TVS can be chosen to be a LCTVS. 
\end{thm}

The reader can find the conditions {\rm {\bf (M1)$^{\prime}$}},  {\rm {\bf (M2)$^{\prime}$}}, and 
 {\rm {\bf (M3)$^{\prime}$}}, and the proof of the theorem,  in Section \ref{wiamcsccs}, in the matrix convex case.
Suffice it to say that {\rm {\bf (M1)$^{\prime}$}} asserts that the $\xi^* v \xi$ expressions satisfy a `convex parallelogram identity' (or alternatively that they satisfy they act as they ought', or that $k \oplus k$ commutes with a certain $2 \times 2$ unitary matrix with scalar entries for $k \in K$), 
{\rm {\bf (M2)$^{\prime}$}} asserts that such  expressions determine $v$, and {\rm {\bf (M3)$^{\prime}$}} 
phrases `direct sums' and `compressions' in terms of such  expressions. 
 We postpone the proof of the `respectively' case of this theorem, namely the abstract 
 characterization of Davidson and Kennedy's nc compact convex sets, to the sequel paper.

    Turning to definitions and notation, we will begin with the phrases used above. We write $\bF$ for $\bR$ or $\bC$. 
     The letter $H$ is usually reserved for real or complex Hilbert spaces, and $K$ for a convex or matrix convex set.  We will be considering real and complex vector spaces, which may also be normed spaces, ordered vector spaces, or $*$-vector spaces. 
     Where the usual $i$-$j$-matrix subscripting conflicts with the lettering for the imaginary number $i$, we sometimes write $\iota$ for the latter. 
     An {\em ordered vector space} is a vector space $E$ with a proper positive cone $E_+$. 
      By a {\em $*$-vector space} we mean a vector space with an involution 
(a period 2 automorphism) $*$.  If $\bF = \bC$ then we assume that the involution is conjugate linear. We write $E_{\rm sa} = \{ x \in E : x = x^* \}$ for the elements in a set $E$ which are selfadjoint.  
We assume that the positive cone of an ordered $*$-vector space $E$ is contained in  $E_{\rm sa}$. 
 We recall that the positive cone is called {\em generating} if $E_+$ spans $E$; e.g.\ if $E_{\rm sa} = E_+ - E_+$. 
    A real (resp.\ complex) Hausdorff LCTVS with a continuous    involution will be called a $*$-LCTVS. Similarly for TVS's. 
    We will say that an ordered $*$-vector space $X$ has a {\em linear base} $K$, if $K$ is a base for $X_+$ (that is, every nonzero $x \in X_+$ may be written uniquely as $ck$ for $c > 0$ and $k \in K$), and $K$ spans $X$. 
    For basics on bases and base norm spaces see e.g.\ the prequel \cite{BH}, which we expect the reader to be somewhat familiar with, at least for the many parts of our paper concerning base norm spaces and their nc variant.
    We write $M_n(\bF)$ for the  $n \times n$ matrices, or sometimes simply $M_n$ when the context is clear.  We write $M_n(E)$ for the $n \times n$ matrices with entries from a vector space $E$. 
    We recall that  $*$-vector space $X$ is {\em matrix ordered} if there is a proper cone $M_n(X)_+ \subset M_n(X)_{\rm sa}$ for all $n \in \bN$, such that $\alpha^* M_n(X)_+  \alpha 
\subseteq M_m(X)_+$ for $\alpha \in M_{n,m}(\bF)$. 
It follows that $M_n(X)_+  \oplus M_m(X)_+ \subset M_{n+m}(X)_+$.  A sequence $(C_n)$ of cones which satisfy these properties is also referred to as a {\em matrix cone}.

We assume that the reader is familiar with basic convexity theory.  We denote the convex hull of a set $A$ by ${\rm co} (A)$ and the closed convex hull by $\overline{{\rm co}} (A)$. 
We should warn the reader about a potential  source of subtle confusion in the literature.   Namely sometimes the naked 
phrase ``compact convex set'' in the literature really means ``compact convex set in an LCTVS'', as opposed to 
 ``compact convex set in a TVS''. This is sometimes not said explicitly, for example in some statements of Kadison's theorem 
 in the literature concerning the duality of the categories  of `compact convex sets' and `function systems'.  They are not the same, not every compact convex set  $K$ in a TVS is affinely homeomorphic to a compact convex set in a LCTVS.  Indeed this was an open problem until Roberts' 1977 needle-point space  counterexample, see e.g.\ \cite{Rob}.  For such $K$, $A(K)$ does not separate points of $K$, which can be a  source of deep trouble. 
  Another  implication of this for us is that not every base norm space with compact base $K$ 
 is a dual base norm space  (Proposition \ref{u}).  To be a dual base norm space one must add the topological
 condition  that $K$  is locally convex.

{\em Operator systems} are a far-reaching noncommutative generalization of function systems which plays a central role in the theory of operator algebras and noncommutative functional analysis, generally.      They may be characterized abstractly as the {\em archimedean matrix order unit spaces},
namely a matrix ordered   $*$-vector space in the sense above which possesses a `matrix order unit' (see \cite{CE,BR} for the latter definition). 
For  $n \in \bN$ we define the amplification of a linear map $T: V \to W$ by 
        \[T^{(n)}: M_n(V) \to M_n(W)\]
        \[[x_{ij}] \mapsto [T (x_{ij})] . \]
        The natural morphisms between matrix ordered spaces (resp.\ operator systems) are the  {\em  completely positive} (resp.\ 
        {\em  unital completely positive} (ucp)) maps.
        These  are linear maps $T: V \to W$ with  every amplification  positive (which implies that it is  selfadjoint, that is $*$-preserving, that is
        $T(x^*) = T(x)^*$ for $x \in V$).  
A (concrete)  operator space is a linear space $E$ of operators on a Hilbert space $H$, together with norms on $M_n(E)$ inherited from $B(H^{(n)})$.
These may be characterized as the vector spaces $E$ with a norm on $M_n(E)$ for all  $n \in \bN$ satisfying certain 
conditions which we will not repeat here. The completely bounded norm is $\| T \|_{\rm cb} = \sup_n \, \| T^{(n)} \|$, and $T$ is completely  contractive (resp.\ completely  isometric)  if  $\| T \|_{\rm cb}  \leq 1$ (resp.\ each $T^{(n)}$ is an isometry).
 A \emph{state} (resp.\ \emph{matrix state}) on an operator system  $V$ is a (selfadjoint) completely positive unital scalar valued ($M_n$-valued) map.  Then $({\rm ucp}(V, M_n))$ is the \emph{matrix state space}. 
 
For general background on complex operator systems and spaces, and in particular on the definitions etc.\ in this section, we refer the reader to e.g.\ \cite{Pnbook,BLM,ERb}.  For real operator systems and spaces see \cite{BR,BReal}. The theory of complex   $C^*$- and von Neumann algebra theory may be found in e.g.\ \cite{Ped}. 
The connection between complex operator systems and matrix convex sets   may be found in \cite{WW} (see also e.g.\ 
\cite{Far,DK,Dav,Fetal} and references therein).  
A matrix  set or nc set in a vector space $E$ is a sequence $(X_n)$, with $X_n \subseteq M_n(E)$.
  We call $X_n$ the $n$th level of $X$.  A set with levels is sometimes called a {\em graded set}.
A (real or complex) {\em matrix convex set} in a (real or complex)  
vector space $E$ is a nc set $K = (K_n)$ in $E$ satisfying 1)\ $x \in K_m$ and $y \in K_n$ implies $x \oplus y \in K_{m+n}$, and
  2)\ $a \in M_{n,m}(\bF)$ with $a^* \, a = I_n$ and $x \in K_n$ implies $a^*xa \in K_m$.   Here $n, m \in \bN$. 
A {\em matrix convex combination}  is a finite sum 
$\sum_i\, \gamma_i^* v_i \gamma_i$ for  $\gamma_i \in M_{n_i,n}$ with $\sum_i\, \gamma_i^*  \gamma_i = I_n$ and $v_i$ at level $n_i$.
So  matrix convex sets are the nc sets which are closed under matrix convex combinations. 
  If $E$ is a topological vector space, then we say $K$ is closed (compact) if each $K_n$ is closed (compact).
  The matrix state space   $({\rm UCP}(\cS, M_n))$ of an operator system $\cS$ is the generic example of a (locally convex) compact matrix convex set \cite{WW}. 
  We write $A(K)$ or $A(K,\bF)$ for the continuous affine scalar functions on a compact convex set $K$, which are unital selfadjoint subspaces of $C(K, \bF)$, the continuous functions on $K$ with values in the field $\bF$. We write $\bA(K)$ or $\bA_{\bF}(K)$ for the noncommutative version from \cite{WW}, the matrix affine continuous nc functions into the scalars.  
  We recall that a nc function is {\em matrix affine} if it preserves matrix convex combinations. 
  Such $\bA(K)$  turn out to be  the generic example of (complete) operator systems  \cite{WW}.  These results give dualities between the categories of operator systems and weak* compact matrix convex sets in an LCTVS (see \cite[Proposition 3.5]{WW}, \cite{DK}), which generalize the `Kadison duality' between 
  function systems and compact convex sets in an LCTVS.

In classical functional analysis, {\em base norm spaces} appear as the objects that are dual to {\em archimedean order unit spaces}, or equivalently, by Kadison's theorem, to the class of function systems.  Both are ordered vector spaces whose order structure induces a norm. Whereas archimedean order unit spaces carry a norm which is induced by an order unit, the norm on a base norm space is induced by a base for the cone of positive elements. 
Similarly in the nc case: nc base norm spaces are the objects that are dual to operator systems.  As we said earlier we expect the reader to be somewhat familiar with 
basics of base norm spaces and their nc variant  in \cite{BH}.  To avoid clutter in the Introduction, we will review the basic definitions and facts about base norm spaces and nc base norm spaces at the start of  Subsection \ref{wiaccu}
and Section \ref{mcs} respectively.

\section{What is a convex set?  Compact convex set?  Matrix convex set?}  \label{wiaccc}

\subsection{What is a convex set?}   \label{wiac} Already in 1939 M. H. Stone had the idea of abstractly characterizing convex sets \cite{Stone}.  In 1954 Hausner gave a much cleaner characterization and proof
 \cite{Hausner}.   Since then many mathematicians and mathematical physicists have needed and have rediscovered this abstract 
  characterization of  convex sets, often much less elegantly than in \cite{Hausner}.       
    
 The abstract `convexity operation' or `mixture'   $K \times K \times [0,1] \to K$  is usually written as $xty$ for $x, y \in K, t \in [0,1]$. A map is {\em abstract affine} if 
 $f(xty) = f(x)tf(y)$  for $x, y \in K, t \in [0,1]$.   Mixtures play a role in the theory of utilities associated in part with von Neumann and Morgenstern (e.g.\ see the  volume containing \cite{Hausner} for more details), and have also been of interest in certain areas of mathematical physics, e.g.\ see references in \cite{BH}. 
 There are four Stone-Hausner  axioms on this `convexity operation' or `mixture'.  Three 
 are the obvious ones: 1)\ $xtx = x$ (self-combination), 2)\ $xty = y (1-t)x$ (commutativity), 3) If  $t \in (0,1]$ and $xty = zty$ then  $x = z$ (cancellation law).
  The final axiom is an `associative law', namely 
 4)\  $(xty)sz = xp(yrz)$ for the obvious scalars $p,r$.  If these four axioms hold then the  Stone-Hausner theorem asserts that $K$ is abstract affinely isomorphic to a convex subset of a 
  vector space. Hausner's proof is to show that 1)\ any abstract convex set $K$ (Hausner calls these mixture spaces) is `affinely' embedded in an `abstract  cone' $\cC$, and 2)\ any abstract  cone may be
  `affinely' embedded in a vector space $X_K$ as a concrete cone.   Indeed for 1), $\cC$ is $K \times (0,\infty)$ with one point adjoined, namely the `zero element' of the cone. 
  The rough or guiding idea for  2)\ is to note that a cone is a semigroup under addition in such a way that the enveloping group (the Grothendieck group of the semigroup), is a vector space.    (We remark that the axioms in the characterization of convex sets 
  in \cite{Gud} for example are different but include the very strong requirement that affine maps separate points, which 
  makes the proof rather trivial, although somewhat interesting.)

  \subsection{The universal vector space $X_K$ of a convex set}  \label{wiacu}  If $K$ is a convex set let $\cC$ be its `generated cone' above, namely $K \times (0,\infty)$ with one point adjoined, namely the `zero element' of the cone.  This has base $K$, and an obvious well defined addition  $sk + t k' = (s+t) [\frac{s}{s+t} k + \frac{t}{s+t} k']$, for $s,t \geq 0$ and $k, k' \in K$.    Let $X_K$ be the Grothendieck group of $\cC$.  Then $(X_K,\cC)$ is an ordered vector space and $X_K = \cC - \cC$.  
   If we desire a complex vector space generated by $K$ then one takes the universal complex vector space of $K$ to be the complexification $(X_K)_c$ of $X_K$.  This is a $*$-vector space with the obvious involution (making $X_K$ the selfadjoint part). 
  The base function of $X_K$ (or its complexification) is the extension in Lemma \ref{a}
 below  of the constant 1 function on $K$.  It   takes $x \in \cC \setminus \{ 0 \}$ to the unique scalar $c$ such that $x/c \in K$.  
 
  A \emph{hyperplane} in a vector space $V$ will be a set of the form $\{x \in V :  f(x) = 1\}$ for a linear functional $f$ on $V$.
  It is an exercise that  a selfadjoint subset $K$ of a $*$-vector space  lies in a real hyperplane not passing through 0 iff it lies in a complex hyperplane not passing through 0. 
   There is a well known trick to replace a convex set $K$ in a vector space $V$ which is in no hyperplane not passing through 0, by an affinely isomorphic convex set which is in such a hyperplane. 
Namely, replace $K$ by $K \times \{ 1 \}$ in $V \oplus \bF$.  

We have the following universal property:
  
  \begin{lemma} \label{a} Let $K$ be a convex set. \begin{enumerate}
\item [{\rm (1)}] Any affine map  $f : K \to V$  into a real vector space $V$, has a unique real linear extension $\tilde{f} : X_K \to V$. 
 If further $V$ is an ordered space and $f$ is positive then so is $\tilde{f}$. 
\item [{\rm (2)}] Any affine map  $f : K \to V$  into a complex vector space $V$, has a unique complex  linear extension $\tilde{f} : (X_K)_c \to V$. 
Also $\tilde{f}$ is selfadjoint if $V$ is a $*$-vector space and $f$ maps into $V_{\rm sa}$. 
\item [{\rm (3)}] If $f$ is  one-to-one and maps into a hyperplane not passing through 0 then the extension $\tilde{f}$  is one-to-one
(in the complex case we assume here that $V$ is a $*$-vector space and $f$ maps into $V_{\rm sa}$).  If in addition the cone $\cC$ in $V$ coming from $f(K)$ (that is $\cC = \bR_+ \, f(K)$) generates $V$, so that  $V = \cC - \cC$, then  $\tilde{f}$ is a (surjective) isomorphism onto $V$.   \end{enumerate} 
   \end{lemma}

\begin{proof}    Indeed define the extension to take $c x - d y$ to $c f(x) - d f(y)$  for $x,y \in K, c,d \geq 0$.  We leave it to the reader (this is a well known 
trick in convexity, see e.g.\ the proof of Lemmas 2.1 or 2.4 in \cite{Breg}) that in both the real and complex case  this is well-defined, 
and is one-to-one in  (3).  
  Indeed we leave the rest as an exercise. 
 \end{proof}

{\bf Remark.}  If $f : K \to V$ is  one-to-one we cannot assert that an extension $X_K \to V$ is one-to-one, unless $f$ maps into a hyperplane not passing through 0.    For example, consider the map
$K \subset l^1_2(\bR) \to \bR$  taking $(t,1-t) \mapsto 2t-1$ for $t \in [0,1]$.   Here $K = \{ (t,1-t) : t \in [0,1] \}$. The extension is  not even  one-to-one on the cone $\cC = \bR_+ K$ generated by $K$.  In this example $X_K \cong l^1_2(\bR)$. 

\bigskip

Not only is $X_K$ universal, but by (3) in the Lemma there is a copy of it (resp.\ its complexification) in every vector space (resp.\ $*$-vector space) $V$ in which $K$ lies in a hyperplane not passing through 0 in $V$ 
(resp.\ in $V_{\rm sa}$).  Thus we will often identify a convex set with its affinely isomorphic image in the vector space $X_K$ (or its complexification).

\subsection{What is a  compact convex set?}   \label{wiacc}  Now that we understand abstract convex sets, one may ask for a topological variant of the Stone-Hausner theorem, namely an abstract characterization of compact convex sets in a TVS. Or, in an LCTVS.   In view of the Stone-Hausner theorem henceforth we have no need of abstract convex sets or abstract cones, 
but simply consider a convex set $K$ which is also a compact Hausdorff topological space.   
  
We say that a convex set is {\em topologically convex} if  it has a Hausdorff topology with respect to which the convexity operation  $K \times K \times [0,1] \to K$ is continuous.

\begin{thm} \label{lmjot}  {\rm (Lawson, Madison, Jamison, O'Brien and Taylor \cite{Law,LM,JOT})}\  If $K$ is a compact topologically convex set then  $K$ may be affine homeomorphically embedded as a compact  convex set in a Hausdorff 
TVS.  If further $K$ is locally convex (that is, has a basis of convex neighborhoods) then 
 $K$ may be affine homeomorphically embedded as a compact  convex set in a Hausdorff 
LCTVS.   \end{thm}

These authors of course built  on earlier work of Klee, Keimel, and many others.  Below we will treat these results  within the framework of base norm spaces (see Lemma \ref{p} and \ref{r}).  This answers the question ``What is a  compact convex set?'' as: it is  an abstract convex set which is compact and topologically convex.  Moreover 
such a set `is' a compact convex set in an LCTVS if and only if it is locally convex. 

 If $K$ is a compact topologically convex set then the cone $\cC$ above generated by $K$ (see the first lines in Subsection \ref{wiacu}),  has a canonical Hausdorff topology with respect to which it is easy to see that it is locally compact and $\sigma$-compact, and the 
 cone operations (addition and multiplication by a nonnegative scalar) are continuous.  
See \cite[Proposition 2.1]{Law}.   Indeed \cite[Corollary 2.3]{Law} says that $K$ is affinely homeomorphic to the compact 
base of a locally compact cone  (with continuous addition and scalar product) in a Hausdorff TVS. 
We give a few more details which are occasionally useful.  Indeed if $L= (0,1] \times K$ with product topology, set $\cC_0$ to be the 
1-point compactification of $L$, with the one point identified with $0$.  
 By the uniqueness of the one-point compactification, $\cC_0$ is topologically identifiable with 
 $\{ x = t k \in \cC : 0 \leq t \leq 1, k \in K \}$, 
 with the topology with basis the sets $I \times B$ for $I = (a,b)$ for $0< a < b$, and $B$ 
 chosen from a basis for $K$, together with the sets $[0,\epsilon) \times K$ for $\epsilon > 0$.  
  Set $\cC_n = [n,n+1] \times K$ for $n \in \bN$, a compact set.   We give $\cC$  the attachment space topology formed by gluing 
$Y = \cC_0$ to $Z = [1,\infty) \times K$ along $A = \{ 1 \} \times K$.   Then $\cC$ is Hausdorff, normal, locally compact, and $\sigma$-compact
(consider the compact $(\cC_n)_{n \geq 0}$ above), 
and $K$ may be identified with $A$ as a compact subset 
and base of $\cC$.    Note that $Z \setminus A = (1,\infty) \times K$ (resp.\ $\cC_0$) may be identified as an open (resp.\ compact) subset 
of $\cC$, and the latter two sets partition $\cC$. 
Indeed $\cC \setminus 0$ is homeomorphic to $(0,\infty) \times K$, the latter with product topology.   

We will call $\cC$ above together with this topology the {\em Lawson cone} of $K$.   We give it this name
even though its construction  is not attributable to Lawson, to honor 
Lawson's beautiful work (partly with coauthors).   

\begin{lemma} \label{o}  If $K$ is a compact topologically convex set then 
the base function on the Lawson cone $\cC$ of $K$ is continuous
on $\cC$.  
 \end{lemma}

\begin{proof}   Suppose that $c_t k_t  \to c k = x$ in $\cC$, where $c_t, c \geq 0, k_t, k \in K$.    
 If $x \neq 0$ then $x \in I \times K$ for some compact interval $I$ in $(0,\infty)$, so that without loss of 
 generality $(c_t)$ is bounded away from $0$.   Since  $\cC \setminus 0$ is homeomorphic to $(0,\infty) \times K$ it follows that
 $c_t \to c$.  Similarly if  $x = 0 \in [0,\epsilon ) \times K$ for $\epsilon > 0$ then $c_t k_t \in [0,\epsilon ) \times K$ eventually. 
 So $c_t \to 0$.  
\end{proof}  

The mechanics in the last proof are often used in what follows. 

  \subsection{The universal Banach space $X_K$ of a compact convex set}  \label{wiaccu} 
  
  We will now need to get into base norm spaces, so we give some more definitions here (see \cite{BH} for a survey  in the classical case, and many more details.)
  A real base norm space is  an ordered normed space $X$ with closed cone $X_+$, with a  convex base $K$ in Ball$(X) \cap 
 X_+$, such that    Ball$(X)  \subseteq t \, 
{\rm co} (K \cup (-K) )$ for all $t > 1$. A real  {\em dual base norm space}  is
a real base norm space with a Banach space predual such that the base $K$ is 
weak* closed (and hence weak* compact).  It follows that Ball$(X)  = 
{\rm co} ( K \cup -K )$, and that the base function (that is, the unique functional on $X$ which is $1$ on $K$) is weak* continuous.    
The dual Banach space of a real unital function space is the generic real  dual  base norm space.   Equivalently, the dual  base norm spaces are exactly (up to appropriate isomorphism) the spaces $A(K)^*$ for a compact convex set $K$
in an LCTVS.  The dual  base of $A(K)^*$ is $\delta(K)$, where $\delta : K \to A(K)^*$ is the canonical map. The generic real base norm space ``is" the predual of a dual real function system, with the base 
corresponding to the normal state space.   Similar results hold for complex base norm spaces and their duality with complex function systems
\cite[Section 3]{BH}.   This is based on the complexification of a real base norm space using the \emph{dual Taylor norm} which we describe there.  
Indeed the theory of complex base norm spaces is a simple add-on to the classical theory of real base norm spaces. Essentially everything in the complex theory follows quickly from the 
real theory, together with the fact above that the complex base norm is completely determined, via the dual Taylor norm,  by the classical `base norm' on the selfadjoint part.  Or in other words, the complex case is just a standard complexification of the real case. 
A base morphism is a positive linear map $u : X \to Y$ between base norm spaces mapping base into base.  This is equivalent (assuming $u$ positive) to $f_Y \circ u = f_X$, where $f_X$ and $f_Y$ are the base functions.  
A bijective base morphism between base norm spaces is an isometric order isomorphism by Lemma \ref{bne}.  Similarly 
a  bijective nc base morphism between nc base norm spaces is a completely isometric order isomorphism. 

  If $K$ is a convex set then there is a canonical seminorm $\| \cdot \|_K$ that one can put on  its universal space $X_K$ from the  earlier subsection, namely the Minkowski functional of co$(K \cup (-K))$.  If this is a norm then
  $X_K$ is a pre-base normed space in the sense of \cite[Section 2]{BH}.   Indeed it is a norm under the conditions of the next result. 
  As in the last subsection, in the following $K$ is  a compact  topologically convex set, 
for example, a compact convex set in a TVS.  Then 
 the Lawson cone $\cC$ above  is Hausdorff, locally compact, and $\sigma$-compact, and has $K$ as a base.  
 Suppose that $K$ has a basis $\cB$ of open (or open and convex) sets.  Then sets of form $I \cdot B$, for $I = (a,b)$ for $0< a < b$, and $B \in \cB$, is a basis for $\cC \setminus 0$, while at $0$ we take $[0,\epsilon) \cdot K$ to be the local basis.  
The Grothendieck group $X_K$ of $\cC$ is as we said earlier, a real ordered vector space with positive cone $\cC$, and base $K$ for $\cC$.  Of course $X_K = \cC - \cC$.

\begin{lemma} \label{p}  If $K$ is  a compact  topologically convex set then $(X_K, \| \cdot \|_K)$ 
is a real base norm space, with norm closed base $K$, and base function 
the unique linear extension of the base function on $\cC$ above.    
Moreover, $X_K$ has a Hausdorff TVS topology extending the topology on $K$.  \end{lemma} 

\begin{proof}  First we show that it is a pre-base norm space
in the sense of \cite[Section 2]{BH}.  For the `linearly bounded' condition, 
suppose that $t \in (0,1), x,y \in K$ and 
$z = tx - (1-t)y \neq 0$, such that for all $n \in \bN$ there exists $s \in [0,1]$ and $x',y' \in K$ with $nz = sx'  - (1-s) y'$.   That is, $ntx + (1-s) y' = sx' + n(1-t) y$.   Then $\frac{nt}{nt + 1-s} x + \frac{1-s}{nt + 1-s} y'$ is in $K$. 
But for some $k'' \in K$ we have 
$$\frac{s}{nt + 1-s} x' + \frac{n(1-t)}{nt + 1-s} y = \frac{s + n(1-t)}{nt + 1-s}  k'',$$  and this is in $K$ iff $s + n(1-t) = nt + 1-s$. 
So $2s + n -1 = 2nt$, or $s = n(t -1/2)+ 1/2$. 
If $t < 1/2$ or $t > 1/2$ we obtain a contradiction for large enough $n$.  If $t = 1/2$ then $s= 1/2$, then we have $\frac{1}{n+1}(nx + y') = 
\frac{1}{n+1}(x' + n y)$.  Write $y' = y_n, x' = x_n$, and choose converging subnets of these with limits $y''$ and $x''$.  Using `continuity of convexity' in $K$ we have $x = y$, and so $z = 0$, a contradiction.   Thus $X$ is a normed space with the base norm, and is 
a pre-base norm space.   

Suppose $k_n \to x = cy - dz$ in norm, where $k_n, y, z  \in K,$ and $c, d \geq 0$.   Then $\| k_n - x \| \to 0$. So we can write $k_n - cy + dz = c_n x_n - d_n y_n$ with
$c_n, d_n \to 0$.   Thus $k_n  + dz + d_n y_n =  cy + c_n x_n$. As before we divide by a scalar so that the last equality becomes a convex combination in $K$.   Suppose that a subnet  $k_{n_t} \to v \in K$ in the topology on $K$.  
Since $c_n, d_n \to 0$, in the limit we have 
$v + dz  =  cy$, using that $K$ is topologically convex.   So $v = x \in K$.   So $K$ is closed, and moreover the topology on $K$ is coarser than the norm topology. 
A similar argument works if we replace  $k_n$ above by $t_n k_n$ with $t_n \geq 0$. Note that applying the base function shows that $(t_n)$ is convergent to $t$ say, so bounded. We obtain $x = cy-dz = tv \in \cC$.  So $\cC$ is closed.  Thus $X_K$  is a 
base norm space with base $K$.  
The assertion about the TVS topology follows from \cite[Theorem 3.2]{LM} and its proof.  This topology is 
the quotient topology on $X$  induced by  the `subtraction'  map $q : \cC \times \cC \to X$, where $\cC$ is
the Lawson cone.  (In the LCTVS case see also the statement of Theorem \ref{llm}, or the first paragraph of the 
proof of Lemma \ref{r}.) 
\end{proof} 

  If we desire a complex  space generated by $K$ then one takes the universal complex normed space of $K$ to be the dual
  Taylor   complexification  of $X_K$.   This may be identified with the projective tensor product 
  $X_K \hat{\otimes}_{\bR} l^2_2(\bR)$.  We write this as $(X_K)_c$ or sometimes simply as $X_K$ when there is no confusion. 
  This is a complex base norm space with norm closed base $K$, if  $K$ is  compact  topologically convex \cite{BH}, and is also  
  a $*$-TVS with the `product topology', whose selfadjoint part is $X_K$.
  
The idea of a universal Banach space $X_K$ of a convex set, and many of the  ideas in the present subsection of our paper,
 have  appeared in various forms or disguises in the literature over 
the decades.   See e.g.\ \cite{Gud} for a representative example.  
We have the following universal property of $X_K$:
  
  \begin{lemma} \label{ab} Let $K$ be a compact topologically convex set, or more
  generally a convex set for which $(X_K, \| \cdot \|_K)$ is a base normed space. \begin{enumerate}
\item [{\rm (1)}] Any bounded affine map  $f : K \to V$  into a real normed space $V$, has a unique bounded real linear extension $\tilde{f} : X_K \to V$, and we have
$\| \tilde{f} \| =  \| f \|_\infty$.   
\item [{\rm (2)}] Any bounded affine map  $f : K \to V_{\rm sa}$  into a complex $*$-normed space $V$, has a unique complex  linear bounded extension $\tilde{f} : (X_K)_c \to V$, with 
$\| \tilde{f} \| =  \| f \|_\infty$.  
\item [{\rm (3)}] If $(V,K')$ is a (real or complex) base norm space and $f : K \to K'$ is affine then $\tilde{f}$ is a (contractive positive) base morphism.
 \item [{\rm (4)}]  If $K$ is compact topologically convex and $f : K \to V$  is affine and continuous into a 
 Hausdorff TVS $V$, 
 then its unique extension $\tilde{f}$ into $V$ from  Lemma {\rm \ref{a}} (in both the real and the complex cases there) is also continuous (with respect to the TVS topology in Lemma {\rm \ref{p}} and the remark after it). 
 \item [{\rm (5)}]  $A_{\rm b}(K,\bR) \cong X_K^*$ and $A_{\rm b}(K,\bC) \cong ((X_K)_c)^*$ isometrically and order isomorphically via a map taking the identity $1$ to the base function for $K$.   Here $A_{\rm b}(\cdot)$ is the bounded affine functions. 
 \end{enumerate} 
   \end{lemma}

\begin{proof}  (1)\  The extension $\tilde{f}$ is done in Lemma \ref{a}. Note that $\|  \tilde{f} \| \geq \| f \|_\infty$ since any element of $K$ has norm 1 in $X_K$. We have $\| c f(k) - d f(k') \|  \leq (c + d) \| f \|_\infty$.
Taking an infimum we have $\| \tilde{f} (x) \|  \leq \| x \| \| f \|_\infty$.  Thus $\| \tilde{f} \| =  \| f \|_\infty$.  

(2)\ Let $\tilde{f} = g_c$ where $g$ is the extension in (1).  Then $\| \tilde{f} \| = \| g_c \| \leq \| g \| =  \| f \|_\infty$, as may be seen  using (1) at the start 
of \cite[Section 3]{BH}.  The reverse inequality is as in (1), so that  $\| \tilde{f} \| =  \| f \|_\infty$.  

(3)\ If $f : K \to K'$ then  $\tilde{f}$ is a base morphism.  Any base morphism  is  contractive and positive.

(4)\ In the real case, since the TVS topology on $X_K$ is a quotient topology, it suffices to check that 
the map $\cC \times \cC \to V : (ck,dk') \mapsto cf(k) - df(k')$, is continuous on the product of the Lawson cone $\cC$.
So suppose that $c_t k_t \to ck$ and $d_t k'_t \to dk$ in $\cC$.  Here $c_t, d_t, c, t \geq 0,$ and $k_t, k_t', k, k' \in K$. If both $c$ and $d$ are nonzero then we have $c_t \to c, d_t \to d, k_t \to k, k'_t \to k'$, as in the proof of Lemma \ref{o}.  Since $V$ is a TVS we have  $c_t f(k_t) - d_t f(k'_t) \to cf(k) - df(k')$ as desired.
Similarly for the other cases.  E.g.\ if $c = 0$ but $d \neq 0$ then $c_t \to 0$ as in the proof of Lemma \ref{o}.  So again 
$c_t f(k_t) - d_t f(k'_t) \to cf(k) - df(k')$ as desired.

In the complex case, the real case shows that the extension $X_K \to V_{\rm sa}$ is continuous.  Hence its complexification
is continuous $(X_K)_c \to V$.

(5)\ The canonical (restriction to $K$) map $X_K^* \to A_{\rm b}(K,\bR)$ is a surjective isometry by (1) with $V = \bR$. It is easy to see that it is an order isomorphism 
taking the base function to 1.  Similarly in the complex case. 
 \end{proof} 

Thus if $K$ is a compact convex set in a real TVS (resp.\ normed space) $V$ then there is a linear (resp.\ and continuous) map $X_K \to V$ onto $\bR_+ K - \bR_+ K$. 
If $K$  lies in a hyperplane not containing $0$ then as we saw in Lemma \ref{a} (3), this map is an isomorphism, so that $V$ contains a base norm space with base $K$.  
Thus not only is $X_K$ universal, but there is a copy of it in every TVS (resp.\ $*$-vector space) $V$ in which $K$ lies in a hyperplane not passing through 0 in $V$ 
(resp.\ in $V_{\rm sa}$).   Thus we will often identify a compact convex set with its affinely homeomorphic (by the last assertion of Lemma \ref{p}) image in the base norm space and TVS 
 $X_K$ (or its complexification).

\begin{lemma} \label{bne}  Suppose that $(X_i,K_i)$ are real or complex base norm spaces and that $K_1$ is affine 
isomorphic to $K_2$.  Then 
$X_1 \cong X_2$ isometrically and as base norm spaces.  Moreover this isomorphism is also a homeomorphism (resp. weak* homeomorphism) if in addition the $K_i$ are compact and topologically convex 
(resp.\ if $(X_i,K_i)$ are dual base norm spaces) and the isomorphism $K_1 \to K_2$ is continuous. \end{lemma}

\begin{proof}   The first assertion  and the first homeomorphism result 
follows from Lemma \ref{ab}.  By the duality of compact convex sets and  function 
 systems, 
$A(K_1) \cong A(K_2)$ as function systems.  Since $X_i = A(K_i)^*$, we have $X_1 \cong X_2$ isometrically as dual base norm spaces. \end{proof}

  Lemma \ref{p} embeds a compact topologically convex set in a (canonical) TVS.    It is natural to ask when it may be embedded in an LCTVS, or when 
  $X_K$ is an LCTVS.   This was answered in \cite{JOT,Law}. 
 The following works in both the real and complex case:
  
  \begin{lemma} \label{r}  If $K$ is a compact topologically convex set which is  locally convex  then $A(K)$ separates points of $K$, and indeed (or equivalently) 
 $K$ is affinely homeomorphic to a compact convex set in a Hausdorff LCTVS.  The latter may be taken to be $A(K,\bR)^*$. 
 Indeed  $X_K$ and its complexification $(X_K)_c$ are dual base norm spaces, and have a Hausdorff LCTVS topology $\tau$ extending the topology on $K$, with respect to which $X_K \cong A(K,\bR)^*$ (resp.\ $(X_K)_c \cong A(K,\bC)^*$) via an isometric 
 base morphism and homeomorphism  with respect to $\tau$ and the weak* topology on $A(K)^*$.  
   \end{lemma}

\begin{proof}  
That   $K$ is affinely homeomorphic to a compact convex set in a  real Hausdorff LCTVS
was shown in \cite[Theorem 1]{JOT}.  Alternatively, in the last section of \cite{Law}, Lawson directly constructs an appropriate LCTVS topology on the vector space $X_K = \bR_+ K - \bR_+ K$.  Indeed he shows that
the quotient topology on $X_K$ induced by the difference map $\bR_+ K  \times \bR_+ K  \to X_K$ is an LCTVS topology extending the topology on $K$.
It follows e.g.\ by the geometric Hahn-Banach theorem in this LCTVS  that $A(K,\bR)$ separates points of $K$.
Hence $A(K,\bC)$ separates points of $K$.

Conversely, in both the real and complex case, once we know that $A(K)$ separates points of $K$ then the canonical map 
$\epsilon : K \to S(A(K)) \subset A(K)^*$ is one-to-one continuous and affine.
So it is  an affine homeomorphism onto a compact convex set in a Hausdorff LCTVS.   Indeed from 
convexity theory, as in (the proof of) `Kadison's theorem', $\epsilon(K) =  S(A(K))$.    Since $A(K)^*$ is a base norm space with base $K$,
Lemma \ref{bne} gives $X_K \cong A(K)^*$ via an isometric 
 base  morphism.   Transferring the weak* topology to $X_K$ gives a topology there with the desired properties. 
(The complex case  can also be seen from the real case and the fact that any real  LCTVS $X$ is real linearly homeomorphically embedded in a
 complex  LCTVS, namely $X \oplus iX$ with the product topology.)  Also, $A(K,\bC)^* = (A(K,\bR)_c)^*
 = (A(K,\bR)^*)_c$, the dual Taylor complexification \cite{BH}.  
 \end{proof} 

{\bf Remark.} See the last assertion of \cite[Corollary 4.2]{Law} for the variant of this where $K$ is locally compact. 

 \begin{theorem} \label{llm} {\rm (\cite[Theorem 5.3]{Law}  and \cite[Theorem 3.2]{LM})}\   If $\cC$ is a locally compact (resp.\  locally compact  and locally convex) (topological) cone then $\cC$ is  embeddable via an
affine homeomorphism as a locally compact cone  in a real Hausdorff  TVS  (resp.\  LCTVS).  Indeed if further  $\cC$ is a cone in a real vector space $X$ with $X = \cC - \cC$, and if $q : \cC \times \cC \to X$ is `subtraction', then 
 $X$ with the quotient topology $\tau$ induced by $q$ is a Hausdorff TVS  (resp.\  LCTVS) for which the canonical inclusion $\cC \subset (X,\tau)$ is a topological imbedding.    
  \end{theorem} 

{\bf Remarks.} \  1)\ The proof of the last result also uses the idea in the proof of \cite[Corollary 4.3]{Law}.\smallskip

2)\ The complex case of the last theorem  follows from the real case, since a real  LCTVS $X$ is real linearly homeomorphically embedded in a
 complex  LCTVS, namely $X \oplus iX$ with the product topology.

\begin{cor} \label{tc}   {\rm (Lawson)}\ If $K$ is a topologically convex set which is compact (resp.\ compact and  locally convex)  then 
 $K$ is affinely homeomorphic to the compact base of a locally compact cone in a Hausdorff TVS  (resp.\ LCTVS). 
 \end{cor}

\begin{proof}  This may be seen via Theorem \ref{llm}, but also follows for example from results in Section \ref{wiaccu} by viewing the Lawson cone of $K$ in $X_K$. \end{proof}

{\bf Remark.} We have geometric Hahn-Banach and separation theorems for a locally compact locally convex (topological) cone $\cC$.  For example, 
if  $x \neq y$ in $\cC$ then there is a continuous linear functional $f$ on $X$ separating them,
where  $X$ is an LCTVS in which $\cC$ is embedded by Theorem \ref{llm}.
  Then $f$ is affine on $\cC$, and separates  $x$ and $y$.
  
   \begin{cor} \label{q}  If a 
topological cone $\cC$ has a compact   base $K$ such that 
$\cC \setminus \{ 0 \} \cong (0,\infty) \times K$ homeomorphically,
then $\cC \cong X_+$ via an affine zero-preserving homeomorphism, for a base norm  space $X$ with base $K$. 
 \end{cor}

\begin{proof}   As in the discussion in the second paragraph of Section \ref{wiacc}, by the uniqueness of the one-point compactification the topology on $\cC$ is
uniquely determined.  Indeed $\cC$ is exactly the Lawson cone of $K$, and is locally compact. The proof of Lemma \ref{p} gives  a base norm space $X$ with $X = \cC - \cC$ and base $K$.  
 The affine zero-preserving map is the one determined by the base. \end{proof} 

{\bf Remarks.} 1)\ In Corollary \ref{q}  we see in fact that $\cC$  is topologically embedded in $X$  as a closed subset.  
See also Section 3.12 in \cite{Jam}. 

\smallskip

2)\ If a  topological cone $\cC$ is locally  compact  and locally convex then it is automatic that it has a compact base 
with $\cC \setminus \{ 0 \} \cong (0,\infty) \times K$ homeomorphically (and in this case $X$ in the last result is a dual base norm space).  To see this note that 
by Corollary \ref{tc}  we can assume that $\cC$ is a locally  compact   cone in an LCTVS $E$.   However it is well known that in this case
$\cC$ is closed in $E$ and has a compact base \cite[Section 3.12]{Jam}.   Indeed the proof of this shows that the base function $\varphi$  is continuous.
Since $K$ is locally convex $X_K$ is a dual base norm space, and we have a continuous one-to-one map  from  $X_K \to E$ taking $\bR_+ K$ onto $\cC$,  which is a homeomorphism on $K$.
 Suppose that $c_t k_t \to c k$ in $\cC$, for $c_t \geq 0, k_t \in K$.  Applying $\varphi$ we see that $c_t \to c$.  If $c \neq 0$ then $k_t \to k$. So $c_t k_t \to ck$ in $X$. 
So we have a homeomorphism between $\cC  \setminus \{ 0 \}$ and $(0,\infty) \times K$ with product topology.  As in the proof of Corollary \ref{q} we have that $\cC$ is exactly the Lawson cone of $K$, and
$X$ is a dual base norm space by Lemma \ref{r}.

 \section{What is a  matrix convex set? A  matrix compact convex set?} \label{wiamcsccs}  

\subsection{What is a  matrix convex set?} \label{wiamcs}  In Subsection \ref{wiac} we discussed the Stone-Hausner abstract characterization of convex sets,
and in Subsection \ref{wiacc}  we discussed the  abstract characterization of compact convex sets.  Now that we understand these, one may ask for an abstract characterization of matrix (or nc) convex sets $K = (K_n)$, and of 
matrix (or nc)  compact convex sets.    One 
may write down a characterization similar to  Stone-Hausner's mixture conditions, that is in terms of multilinear 
maps in variables from $K_n$ and matrix algebras, satisfying certain `mixture conditions' resembling the four  Stone-Hausner axioms.  
However this can be quite clumsy looking if not done carefully.
Moreover in view of the  Stone-Hausner theorem we have no need of abstract convex sets.   We supply a few alternative characterizations.   For example, one approach is to take a perspective 
from the theory of noncommutative functions (see e.g.\ \cite{KVV}), namely that nc sets and nc functions are defined in terms of a particular fixed vector space.  
We will define a (complex) {\em abstract  matrix convex set} $K = (K_n)$ in
 terms of a canonical vector space $Z$ which is determined by the 
abstract convex set $K_1$.   More particularly,  $Z = Z_{K_1}$ is the complexification of the real universal vector space $X_{K_1}$ from  Subsection
\ref{wiacu}.  We say that $K$ is a (complex) abstract  matrix convex set if  it
is a  complex matrix convex set in the usual sense (see \cite{WW}, or our introduction) in the vector space $Z  = Z_{K_1}$.

If one is in a situation where further abstraction is convenient, 
we supply a suitable 
framework in terms of a kind of `quadratic form'.  
Let  $K$ be a sequence of sets $(K_n)$, such that $K_1$ is an abstract convex set.
Suppose that we have  actions $\bC^n \times K_n \to K_1$,  written as $\xi^* v \xi$, for $n \in \bN, v \in K_n, 
\xi \in \bC^n, \| \xi \| = 1$.  We call these {\em  compressions to} $K_1$, or $1$-{\em compressions}, 
and sometimes write it as $v_\xi$.
This action is not defined if $\| \xi \| \neq 1$, but we may extend the notation to
 all $\xi \in \bC^n$ by defining $\xi^* v \xi = \| \xi \|^2 (\xi')^* v \xi'$,  where we write $\eta'$ for the unit vector normalization of
a nonzero vector $\eta \in \bC^n$. Of course $0^* v 0 = 0$. 
We may then talk about the (matrix) {\em entries} of an element
$k \in K_n$.  These are a canonical linear combination in $Z$ coming from the polarization identity
\begin{equation}
    \label{pn} \begin{split} 
k_{ij}  &=  \frac{1}{4} \sum_{k=0}^3 \, \iota^k (\eta + \iota^k \xi)^* k (\eta + \iota^k \xi), \qquad  \eta = e_j, \xi = e_i , \\
k_{ii} &= e_i^* k e_i.  \end{split}
\end{equation} 
(Here $(e_i)$ is the canonical basis of $\bC^n$.) 
Then $[k_{ij}] \in M_n(Z)$.   Let $\tilde{K_n}$ be the set of such matrices in $M_n(Z)$, for $k \in K_n, n \in \bN$. 
We  require that

\medskip

\begin{adjustwidth}{2em}{2em} 
(M1) $ \; \; \xi^* [k_{ij}] \xi = \xi^* k \xi = k_\xi , \qquad n \in \bN, k \in K_n, \xi \in \bC^n,  \| \xi \| = 1. $
 \end{adjustwidth} 
 
\medskip

 \noindent   That is, the compressions to $K_1$ act as they ought.  If one is concerned that this characterization 
 is not phrased intrinsically in terms of $K$, we have a couple of reformulations of (M1) which are
 intrinsic to $K$.  For example:

\begin{lemma} \label{M1M1'}   Equation {\rm (M1)} is equivalent to 
\begin{adjustwidth}{2em}{2em} 
{\rm (M1)$^\prime$} \,   The 1-compressions satisfy the parallelogram identity:
$$ 2(\xi^* k \xi + \eta^* k \eta) = (\xi + \eta)^* k  (\xi + \eta) + (\xi - \eta)^* k  (\xi - \eta), \qquad \xi, \eta \in \bC^n, k \in K_n, n \in \bN.$$ 
 \end{adjustwidth} 
 \end{lemma}
 
(This may be made into an equation inside $K_1$ by scaling so that $\| \xi \|^2 + \| \eta \|^2 = 1$. Then the left side divided by $2$ is a 
convex combination of two elements of $K_1$.   By the usual parallelogram identity 
$\| \xi + \eta \|^2 + \| \xi - \eta \|^2 = 2$, so that the right side divided by $2$ is also a  
convex combination of two elements of $K_1$.)

\begin{proof}  If (M1) holds then by (M1), and by multiplying out the parentheses, $$\xi^* k \xi + \eta^* k \eta = \xi^* [k_{ij}] \xi + \eta^* [k_{ij}] \eta = \frac{1}{2}((\xi + \eta)^* [k_{ij}] (\xi + \eta) + (\xi - \eta)^* [k_{ij}] (\xi - \eta)).$$
This equals  $\frac{1}{2}((\xi + \eta)^* k  (\xi + \eta) + (\xi - \eta)^* k  (\xi - \eta))$ by (M1) again.  
Thus (M1)$^\prime$ holds.

For the converse, let $\varphi$ be a real linear functional on $X_{K_1}$.  Applying $\varphi$ to  (M1)$^\prime$ we see that 
$\xi \mapsto \varphi(\xi^* k \xi)$ satisfies the parallelogram identity on $\bC^n$. 
Thus by \cite[Theorem 11.5]{A} there exists a sesquilinear  map  $B : \bC^n \times \bC^n \to \bC$ with
$B(\xi, \xi) =  \varphi(\xi^* k \xi)$.   By the polarization identity it is easy to see that $B(e_i, e_j)$ is just $\varphi_c$ applied 
to the expression in (\ref{pn}), and $B$ is a selfadjoint sesquilinear form.  Indeed with the notation there 
$$\varphi_c(k_{ij}) =  \frac{1}{4} \sum_{k=0}^3 \, \iota^k B(\eta + \iota^k \xi, \eta + \iota^k \xi)
= B(e_i, e_j).$$
Thus if $\xi = [z_i]$ then $$\varphi(\xi^* [k_{ij}] \xi) = \sum_{i,j} \overline{z_i} \varphi_c(k_{ij} ) z_j =
\sum_{i,j} \overline{z_i} B(e_i, e_j) z_j = B(\xi, \xi) =  \varphi(\xi^* k \xi).$$
Since the linear dual of  $X_{K_1}$ separates points of $X_{K_1}$ we have that $\xi^* [k_{ij}] \xi = \xi^* k \xi$. 
So (M1) holds. \end{proof}  

We will later be able to replace (M1) or  (M1)$^\prime$ with the very mild, even at-first-glance redundant, condition

\medskip

\begin{adjustwidth}{2em}{2em} 
(M1)$^{\prime \prime}$\ For $n \in \bN$ and $k \in K_n$, the element $k \oplus k \in K_{2n}$ commutes with 
$$U = \frac{1}{\sqrt{2}} \, \begin{bmatrix}  I_n &  I_n \\   I_n & - I_n  \end{bmatrix} .$$
That is, $U^* (k \oplus k) U = k \oplus k .$
 \end{adjustwidth} 
 
\medskip

We also require that: 
(M2) \ the entries of $k \in K_n$ determine $k$.  That is, 
$k \to [k_{ij}]$ is one-to-one on each $K_n$ for all $n \in \bN$.  We will check below 
that assuming (M1),  (M2) is equivalent to 

\medskip

\begin{adjustwidth}{2em}{2em} 
(M2)$^\prime$ \, The compressions to $K_1$ determine
 $k \in K_n$.   \end{adjustwidth} 
 
\medskip

 \noindent  That is, if $\xi^* k \xi = \xi^* k' \xi$, for $k, k' \in K_n, 
\xi \in \bC^n, \| \xi \| = 1$, then $k = k'$.  

We say that $K$ is a (complex) {\em abstract  matrix convex set} if (M1) and (M2) hold,
as well as (M3)\ $\tilde{K} = (\tilde{K_n})$ is a  complex matrix convex set in the usual sense (see \cite{WW}), in the vector space $Z$. 
Alternatively to (M3), one may write down the `matrix convexity operations'  for $K$ in terms of the compressions to $K_1$, and insist that $K$ is closed under these.  Namely, we have  {\em general compressions} 
$$(\gamma^* k \gamma)_\xi = (\gamma  \xi)^* \, k  \, (\gamma \xi)
= k_{\gamma \xi}, \qquad \gamma \in M_{m,n}, \gamma^* \gamma = I_n, k \in K_m, \xi \in \bC^n,  \| \xi \| = 1,$$
and {\em direct sum}
$$(k_1 \oplus k_2)_\xi = \xi_1^* k_1 \xi_1 + \xi_2^* k_2
 \xi_2 , \qquad  k_1 \in K_n, k_2 \in K_m , \xi \in \bC^{n+m},  \| \xi \| = 1,$$
where the $\xi_k$ are  the projection of $\xi$ onto $\bC^{n} \oplus \vec 0$ and $\vec 0  \oplus \bC^{m}$ in $\bC^{n+m}$,
and $\xi_j^* k_j \xi_j =  \| \xi_j \|^2  \, (k_j)_{\xi_j'}$ for $j = 1,2$, where we write $\eta'$ for the unit vector normalization of
a vector $\eta \in \bC^n$.
Thus the formulation of the alternative to (M3) is: 

\medskip

\begin{adjustwidth}{2em}{2em} 
(M3)$^\prime$ \, For all $n, m$ and $\gamma$ as above, and $k, k_2 \in K_m, k_1 \in K_n$, there exists an element in $K_n$ written as $\gamma^* k \gamma$, 
and an element in $K_{n+m}$ written as $k_1 \oplus k_2$, such that the last two centered equations hold.  
\end{adjustwidth} 
 
\medskip

\begin{lemma} \label{M1M1''}   If both {\rm (M3)$^{\prime}$}  and {\rm  (M1)$^{\prime \prime}$} hold then so does  
{\rm (M1)$^{\prime}$}. Conversely,  if {\rm (M1)$^{\prime}$}, {\rm (M2)$^{\prime}$}, and {\rm (M3)$^{\prime}$} all hold then so does 
 {\rm (M1)$^{\prime \prime}$}. 
 \end{lemma}

\begin{proof}  Assume that {\rm (M3)$^\prime$}  holds, and let $U$ be as in (M1)$^{\prime \prime}$.
For $k \in K_n$  let $k' = k \oplus k \in K_{2n}$.  If  $\| \xi \|^2 + \| \eta \|^2 = 1$ set $\zeta = (\xi, \eta)  \in \bC^{2n}$. By (M3)$^\prime$ used twice we have 
$$\xi^* k \xi + \eta^* k \eta = \zeta^* k' \zeta = \zeta^* U^* k' U \zeta = \frac{1}{2}((\xi + \eta)^* k  (\xi + \eta) + (\xi - \eta)^* k  (\xi - \eta)),$$ which gives (M1)$^\prime$.

The argument reverses to show that  (M1)$^\prime$ and {\rm (M3)$^\prime$}  imply that 
$\zeta^* k' \zeta = \zeta^* U^* k' U \zeta$ for all $\zeta \in \bC^{2n}$, so that {\rm (M1)$^{\prime \prime}$} holds by 
(M2)$^{\prime}$. \end{proof}  

\noindent An {\em abstract nc affine map}  is a nc function which preserves  the direct sums 
$\oplus$ and compressions defined in {\rm (M3)$^\prime$}.
An {\em abstract  complex matrix convex set} is a sequence $K = (K_n)$  of sets 
which satisfies  the 
conditions in the next theorem, or equivalently which satisfies  {\rm (M1)$^{\prime \prime}$}, {\rm (M2)$^{\prime}$}, and {\rm (M3)$^{\prime}$}.  Or equivalently, by the theorem, which is abstractly nc affinely isomorphic to a concrete 
(complex) matrix convex set in the sense of \cite{WW}. 
 We will adopt a similar but more complicated approach to real matrix convex sets elsewhere.

\begin{thm} \label{defmc}  Let $K$ be a sequence of sets $(K_n)$, such that $K_1$ is an abstract convex set and we have 
actions $\bC^n \times K_n \to K_1$ for all $n \in \bN$ 
satisfying {\rm (M1)} (or equivalently, {\rm (M1)$^\prime$}).  Then $K$ satisfies {\rm (M2)} and {\rm  (M3)} if and only if it satisfies {\rm (M2)$^\prime$} and {\rm (M3)$^\prime$}.
If these all hold then $\tilde{K}$ above is a concrete (complex) selfadjoint matrix convex set in the $*$-vector space
$Z$ and  $K$ is abstractly nc affinely isomorphic to $\tilde{K}$.

Conversely, every concrete matrix convex set  is an abstract matrix convex set. 
\end{thm}

\begin{proof}    Clearly (M2) implies (M2)$^\prime$ since $k_{ij}$ is a fixed linear combination of 
fixed compressions to $K_1$ by (\ref{pn}).  Conversely, if $[k_{ij}] = [k'_{ij}]$ and (M1) and (M2)$^\prime$ hold 
then $k = k'$ since these have the same compressions to $K_1$ by (M1).  
If $M_k = [k_{ij}]$ is the matrix for $k$, then by (M1) the 
first centered equation defining (M3)$^\prime$
may be rewritten as $$(M_{\gamma^* k \gamma})_\xi = (M_{k})_{\gamma \xi} = ( \gamma^* M_k \gamma)_\xi.$$
These hold for all such $\xi$ if and only if $M_{\gamma^* k \gamma} = \gamma^* M_k \gamma$.
Thus this part of (M3)$^\prime$ is equivalent to saying that $\tilde{K}$ is closed under compressions. 
Similarly, the second  centered equation as it is used in (M3)$^\prime$ is equivalent to saying that $\tilde{K}$ is closed under 
direct sum, via the equation: 
$$(M_{k_1 \oplus k_2})_\xi = \| \xi_1 \|^2 (M_{k_1})_{\xi_1'} + \| \xi_2 \|^2 (M_{k_2})_{\xi_2'} 
= (M_{k_1} \oplus M_{k_2})_\xi.$$
 Moreover the map $k \mapsto [k_{ij}]$ is  an abstract nc affine isomorphism $K \cong \tilde{K}$.

Note that $Z$ is a $*$-vector space with the obvious involution $(x +iy)^* = x-iy$ for $x, y \in X_{K_1}$.  Then $K$ is selfadjoint, that is, $\tilde{K}_n \subset M_n(Z)_{\rm sa}$. 
 Indeed $k_{ii}^* = k_{ii}$ since $k_{ii} \in K_1 \subset Z_{\rm sa} = X_{K_1}$.
That  $k_{ji}^* = k_{ij}$ may be easily seen  from (\ref{pn}). 

For the converse, clearly every concrete matrix convex set satisfies {\rm (M1)$^{\prime \prime}$}, {\rm (M2)$^{\prime}$}, and {\rm (M3)$^{\prime}$}.
\end{proof}

A {\em nc hyperplane} in $E$ is the sequence $(H_n)$ of sets $H_n = \{x \in M_n(E): f^{(n)}(x) = I_n\}$, where $f$ is a fixed linear functional on $E$, and $I_n$ is the $n \times n$ identity matrix.  
The hyperplane trick mentioned above Lemma \ref{a} works for a matrix convex set $K$ in $V$ to get a nc 
affinely isomorphic convex set which is in a nc hyperplane not passing through 0: consider $(K_n \times \{ I_n \})$ in  $V \oplus \bF$.   Indeed if $H_1$ lies in a hyperplane in $E$ not containing 0, then $H$ lies in a nc hyperplane not containing 0. 
For if $f$ is a linear functional which is $1$ on $H_1$ then $\xi^* f^{(n)}(k) \xi = f(\xi^* k \xi) = 1$ for all $k \in H_n$ and all unit vectors $\xi \in \bC^n$.  Thus $f^{(n)}(k) = I_n$.

\begin{prop} \label{isaf}
Suppose that $K$ is a  complex matrix convex set in a complex vector space $V$. Then 
 $K$ is matrix affinely embedded as a selfadjoint matrix convex set in a $*$-vector space. 
 Indeed this may be done with the latter contained in a  nc hyperplane not passing through 0. 
 \end{prop}

\begin{proof} Similarly to  the discussion before \cite[Proposition 3.5]{WW}, 
 let $A_{\bC}(K)$ denote the ordered `unital' 
 $*$-vector space of `scalar' matrix affine functions on $K$, with $M_n(A_{\bC}(K))$ identified with
 the $M_n$-valued matrix affine functions on $K$.  Similarly to  the proof of 
 \cite[Proposition 3.5 (2)]{WW}, there is a canonical `evaluation map' $\theta$ from $K$ into the 
 selfadjoint  complex matrix convex set $L = (L_n)$ of $M_n$-valued linear ucp selfadjoint  
 maps on $A_{\bC}(K)$.  
 For $f : K \to M_n$ matrix affine we have $(f^*)_n(k) = f_n(k)^*$ for $k \in K_n$.
Thus $$\theta_n(k)^*(f) =  \theta_n(k)(f^*)^* =(f^*)_n(k)^* =   f_n(k) = \theta_n(k)(f).$$  
Therefore indeed $\theta_n(k)$ is selfadjoint for $k \in K_n$.
It is easy to see as in that reference that $\theta$ is matrix affine.
 Since the linear functionals on $V$ separate points of $V$, and since every such linear functional
 restricts to an element of $A_{\bC}(K)$, it follows as in  the proof of 
 \cite[Proposition 3.5 (2)]{WW} that $\theta$ is one-to-one.  This proves the  first assertion.

 Note that $\theta_1(K_1)$ is contained in the hyperplane defined by $1$ in $A_{\bC}(K)^*,$ 
 and in $(A_{\bC}(K)^*)_{\rm sa}$. Indeed $\theta(K)$  is contained in the nc hyperplane defined by $1$.
 \end{proof}

{\bf Remark.}  We often identify a matrix convex set $K$ with its matrix affinely isomorphic image   $\tilde{K}$ in the vector space $Z$.   Recall that $\widetilde{K_n} = \{ [k_{ij}] : k \in K_n \}$. The proposition leads to  another proof of this.
Indeed suppose, using the proposition, that $V$ is a complex $*$-vector space with $K$ matrix affinely embedded 
as a complex matrix convex set $L$ in $V$.
Thus  there is a one-to-one matrix affine function $\epsilon : K \to V$ with $\epsilon(K) = L$.
We may assume that $L_1$ lies in a hyperplane in $V_{\rm sa}$ not containing 0, by the usual trick.   In this case $L$ lies in a nc hyperplane not containing 0, as we said above the proposition. Lemma \ref{a}  shows that $\epsilon_1$ extends to a 
linear map $\tilde{\epsilon} : Z  \to V$ taking $\tilde{K}$ onto $L$, and this map is a selfadjoint linear isomorphism onto its range if $L_1$ lies in a hyperplane in $V_{\rm sa}$ not containing 0. 
Indeed if $k \in K_n$ then since $\xi^*  \epsilon_n(k) \xi = \epsilon(\xi^* k \xi)$ for all unit vectors $\xi \in \bC^n$, we have 
$$\tilde{\epsilon}(k_{ij}) = \tilde{\epsilon}(\frac{1}{4} \sum_{k=0}^3 \, \iota^k (\eta + \iota^k \xi)^* k (\eta + \iota^k \xi) )
= \frac{1}{4} \sum_{k=0}^3 \, \iota^k  (\eta + \iota^k \xi)^* \epsilon_n(k) (\eta + \iota^k \xi) = \epsilon_n(k)_{ij} ,$$ where $\eta = e_j, \xi = e_i$.  So $\tilde{\epsilon}^{(n)}([k_{ij}]) = \epsilon_n(k)$.  Thus $$\xi^*  \epsilon_n(k) \xi = \xi^* \tilde{\epsilon}^{(n)}([k_{ij}]) \xi =  \epsilon(\xi^* k \xi),$$ giving (M1).
Indeed 
the inverse of $\tilde{\epsilon}$ maps $L$ onto a  matrix convex set in $Z$, which is exactly the nc set $\tilde{K}$ defined after  (\ref{pn}). Of course $(\tilde{\epsilon}^{(n)})$ is a nc affine map, and so is its inverse.

 \subsection{An alternative proof of the matrix convex set characterization}
 
 We first observe that  a noncommutative generalization of the somewhat tautological 
 characterization of convex sets
 in \cite{Gud} follows almost trivially using the $A_{\bC}(K)$  construction in the proof of Proposition \ref{isaf}, or its abstract version.  We give some more details
 (a full account may be found in the second author's thesis).
 A \emph{matrix abstract structure}
on a graded set $K = (K_n)$ consists of maps
$$
T_{\alpha_1,\dots,\alpha_k}:
K_{n_1}\times\cdots\times K_{n_k}\longrightarrow K_n
$$
for every finite family
$\alpha_i\in M_{n_i,n}$ with $\sum_{i=1}^k\alpha_i^*\alpha_i=I_n$ for $n_i\in \bN.$
(By the usual tricks, having a matrix abstract structure is equivalent to being closed 
under  the `abstract direct sum', which are just $T_{\alpha,\beta}$ where $\alpha = [I \; 0 ]$ and $\beta = [0 \; I]$, and under `abstract compressions' $T_{\gamma}$ for  matrix isometries $\gamma \in M_{n,m}$.) 
Let $A(K)$ denote the  space of nc functions 
$$
f=(f_n)_{n\geq 1},
\qquad
f_n:K_n\to M_n,
$$
such that for every $\alpha_1,\ldots,\alpha_k$ and 
$x_i\in K_{n_i}$ as above, we have
$$
f_n\!\left(T_{\alpha_1,\ldots,\alpha_k}(x_1,\ldots,x_k)\right)
=
\sum_{i=1}^k \alpha_i^* f_{n_i}(x_i)\alpha_i.
$$
Elements of $A(K)$ are called (abstract) matrix affine functions on $K$. As  in the proof of Proposition \ref{isaf}, the collection $A(K)$ is a  vector space under pointwise operations.  

Similarly the  (abstract) matrix affine functions $A(K,M_r)$ is a  vector space which may be identified with 
$M_r(A(K))$, as is by now familiar (see e.g.\ \cite{WW}).
For $x\in K_n$, define $\kappa_n(x):A(K)\to M_n$ by $\kappa_n(x)(f)=f_n(x).$ Let 
$\kappa=(\kappa_n) :K\to A(K)^\prime$. Here $A(K)^\prime$ is the vector space dual of $A(K)$ with 
$M_n(A(K)') = {\rm  Lin}(A(K),M_n)$.
The `theorem' then is that if $A(K)$ completely separates points of $K$ then $\kappa:K\to A(K)^\prime$ is matrix affine and injective.  Thus $K$ is abstractly matrix affine 
isomorphic to a matrix convex set in a vector space.  As suggested above, the proof is trivial: $\kappa$ is 
 matrix affine by the definitions, and is injective since $A(K)$ completely separates points of $K$.
 
 We next remark that just as was the case with condition (M2)$^\prime$ above, $A(K)$ completely separating points of $K$ is equivalent to the following pair of conditions holding: a)\ $A(K)$ separating points of $K_1$, and b)\ 
compressions (to level 1) separating points of $K$ (that is, condition (M2)$^{\prime}$ holds).  Indeed 
if  $x\neq y$ in $K_n$, then by b)\ there 
is a unit vector $\xi\in\mathbb C^n$ such that $T_\xi(x)\neq T_\xi(y)$ in $K_1$.  Then 
$f(T_\xi(x)) \neq f(T_\xi(y))$ for some $f \in A(K)$ by a). Hence  $\xi^* f_n(x) \xi \neq \xi^* f_n(y) \xi$, so that $f_n(x)  \neq f_n(y)$.  (We also note that this argument is essentially reversible.  It also demonstrates that 
$A(K)$ completely separates points of a concrete matrix convex set $K$, giving the converse part of the characterization  `theorem' in the last paragraph.) 

This now gives an alternative proof of the characterization of matrix convex sets in Theorem \ref{defmc} 
(or Theorem \ref{defmck}) which the reader might prefer.  That is,  {\rm (M1)$^{\prime \prime}$}, {\rm (M2)$^{\prime}$}, and {\rm (M3)$^{\prime}$} are together equivalent to 
$K$ being abstractly nc affinely isomorphic to a concrete matrix convex set.  Indeed by the last paragraphs
we need only show that $A(K)$ separates points of $K_1$.   Since the linear dual of $X_{K_1}$ separates points of $K_1$,  
 we need only establish the 
Claim:  for each linear $\varphi :  X_{K_1} \to \bR$ there exists $f \in A(K)$ with $f_1 = \varphi_{| K_1}$. To this end, suppose that $k \in K_n$ and $\xi \in \bC^n$.  By the first part of the proof of the converse direction of Lemma \ref{M1M1'} there exists a selfadjoint sesquilinear form $B$ on $\bC^n$
with $B(\xi, \xi) =  \varphi(\xi^* k \xi)$.  By undergraduate matrix theory 
there is a unique selfadjoint matrix $A \in M_n(\bC)$ with 
$$\varphi(\xi^* k \xi) = \langle A \xi , \xi \rangle, \qquad \xi \in \bC^n .$$   Define $f_n(k) = A$, so that 
$\langle f_n(k)  \xi , \xi \rangle = \varphi(\xi^* k \xi)$ for $\xi \in \bC^n$.   Then (M3)$^\prime$
implies that $f$ is matrix affine.   Indeed if $\gamma \in M_{nm}$ is an isometry then
$$\langle  f_m(\gamma^* k \gamma)  \xi , \xi \rangle =  \varphi(\xi^* \gamma^* k \gamma \xi) = 
\langle  f_n(k) \gamma \xi , \gamma \xi  \rangle = \langle  \gamma^* f_n(k) \gamma \xi ,  \xi  \rangle , \qquad \xi \in \bC^n .$$
So $f_m(\gamma^* k \gamma) = \gamma^* f_n(k) \gamma$ by
(M2)$^\prime$. 
    Similarly, $\langle  f_{n+m}(k_1 \oplus k_2)  \xi , \xi \rangle =  \varphi(\xi^* (k_1 \oplus k_2)  \xi)$ for $k_1 \in K_n, k_2 \in K_m$, which equals 
$$\varphi(\xi_1^* k_1 \xi_1 + \xi_2^* k_2 \xi_2) = 
\langle  f_n(k_1) \xi_1 , \xi_1  \rangle + \langle  f_m(k_2) \xi_2 , \xi_2  \rangle = \langle (
f_n(k_1) \oplus   f_m(k_2) ) \xi ,  \xi  \rangle.$$ So $f_{n+m}(k_1 \oplus k_2) = f_n(k_1) \oplus   f_m(k_2)$ by
(M2)$^\prime$. 
Thus $f$ is matrix affine.   Moreover $f_1 = \varphi_{| K_1}$. This proves the Claim.

\subsection{What is a  matrix  compact convex set?} \label{wiamccs}  

We say that an (abstract) matrix convex set $K$  is {\em matrix topologically convex} if there is a Hausdorff topology $\tau_n$ on $K_n$ for all $n$ with respect to which the `matrix convexity operations'
are 
continuous (so $\sum_{k=1}^m \gamma_k^* v_k  \gamma_k$ is continuous in the $\gamma_k, v_k$).    This implies (using e.g.\ (\ref{pn})) that the 
entries $k_{ij}$ are a continuous function of $k \in K_n$ whenever $K$ is a matrix convex set in a TVS (with the topology on $K_n$ being the relative product topology).   
We will see below that $K$ is matrix topologically convex if and only if  $K_1$ is topologically convex
and the $K_1$-compressions $\xi^* k \xi$ are continuous in $k \in K_n$ for all $n \in \bN$ and fixed $\xi \in \bC^n, \| \xi \| = 1$.
Or one may replace the last condition by saying that the entries $k_{ij}$ are a continuous function of $k \in K_n$ for all $n \in \bN$.
 See also e.g.\ Proposition  \ref{isc}  for another example of this trick.

Now suppose that the matrix topologically convex set $K$ in the last paragraph is matrix compact, so that $K_n$ is compact for all finite $n$.  In this case we simply say that $K$ is an {\em abstract matrix compactly convex set}.  The reader should take note of this somewhat subtle usage in the rest of our paper; we really mean by this an  `abstract compact matrix topologically convex set'. 
Since $K_1$ is compact and topologically convex there is a canonical TVS topology $\mathcal{T}$ on the real base norm space 
$X_{K_1}$ and on its complexification $Z$ 
described in Lemma \ref{p} and the remark after it.  Hence there is also a canonical TVS topology $\sigma_n$ on 
$M_n(Z)$ for each $n \in \bN$ (the `product topology' of $n^2$ copies of $\mathcal{T}$).

The first statement of the next theorem yields an abstract  characterization of (complex)  compact matrix  convex sets in a Hausdorff TVS,
as  the  abstract matrix  compactly  convex sets.  It follows from the later Corollary \ref{xkn} that this TVS may be taken to be 
an LCTVS if and only if in addition $K_1$ is locally convex. 

\begin{thm} \label{vae}    If $(K,\tau)$ is a 
 (complex)  abstract matrix  compactly  convex set 
 then $(K,\tau)$ is matrix affinely homeomorphic to its canonical (compact matrix convex set) image  
 $\tilde{K}$ in the TVS  $Z$.
\begin{enumerate}
\item [{\rm (1)}]  If $(K,\tau)$ is a 
 (complex)  abstract matrix  compactly  convex set 
 then the $i$-$j$-entries of $k \in K_n$ 
are continuous maps from $(K_n,\tau_n)$ to $(Z,\mathcal{T})$, for each $n \in \bN$.  
We have $\tau_n = \sigma_n$ on $K_n$ for each $n \in \bN$ (identifying $K_n$ and $\widetilde{K_n}$).    
\item [{\rm (2)}]  A matrix convex set $K$  for which $K_1$ is compact and topologically convex in some topology, 
is an abstract  matrix compactly convex set (with the same level 1 topology) if and only if its canonical image  in the TVS $Z$ is  matrix compact. 
\item [{\rm (3)}]  A compact matrix convex set $K$ in a Hausdorff TVS $V$ is an abstract  matrix compactly  convex set.
If further $V$ is a LCTVS then so is 
$X_{K_1}$
and its complexification $Z$. 
\item [{\rm (4)}]   A (complex)  matrix convex set $K$ which is compact in 
some (graded) topology on $K$, is matrix topologically convex if and only if  $K_1$   is topologically convex
and the compressions $\xi^* k \xi$ are continuous in $k \in K_n$ for all $n \in \bN,$ and fixed $\xi \in \bC^n, \| \xi \| = 1$.
One may replace this last condition here by   the entries $k_{ij}$ being continuous functions of $k \in K_n$ for all $n \in \bN$. 
\end{enumerate} 
 \end{thm}

\begin{proof}  (1) and (2) \  To see that the $i$-$j$-entries of $k \in K_n$ 
are continuous,
let $k^t \to k$ in $\tau_n$.  Then $\xi^* k^t \xi \to \xi^* k \xi$ in $\tau_1$ and therefore also in $\mathcal{T}$, for each  unit vector $\xi \in \bC^n$.  Thus $k^{t}_{ij} \to k'_{ij}$ in $Z$
by Eq.\ (\ref{pn}). 
Clearly $k_t \to k$ in $\sigma_n$.  Conversely suppose that  $k_t \to k$ in $(K_n,\sigma_n)$ and a converging subnet  $k_{t _\mu}\to k'$ in $\tau_n$.
Then $$\frac{1}{2}(\eta + i^k \xi)^* k^{t _\mu} (\eta + i^k \xi) \to \frac{1}{2}(\eta + i^k \xi)^* k' (\eta + i^k \xi), \qquad \eta = e_j, \xi = e_i$$ in $K_1$. 
Thus $k^{t _\mu}_{ij} \to k'_{ij}$ in $Z$, so that $k^{t _\mu}\to k'$ in $\sigma_n$. So $k = k'$ and $\tau_n = \sigma_n$. 
The converse in (2)  is clear from the above. 

(3)\   The first assertion is clear from (1) say.
If $V$ is a LCTVS then $X_{K_1}$ is a LCTVS, as is its complexification $Z$, by Lemma \ref{r}.  

(4)\ The proof of (1) shows that the compressions to $K_1$ are continuous if and only if  the entries $k_{ij}$ are continuous
(the converse clearly following from 
condition (M1) above).  It also shows that $K$ is matrix affinely homeomorphic to $\tilde{K}$, which implies that $K$ is 
topologically convex.  The converse is obvious. 
 \end{proof}

 The following (also sometimes encountered in \cite{DK,WW}) is a useful criterion for when a nc map between matrix convex sets is nc affine or nc continuous:
 
 \begin{prop} \label{isc}  A nc function $g : K \to L$ between complex matrix compact convex sets is matrix 
 affine if and only if $g_1$ is affine and satisfies 
 $\xi^*  g_n(k) \xi = g_1(\xi^* k \xi)$ for all unit vectors $\xi \in \bC^n$ and $k \in K_n, n \in \bN$.  If these hold then $g$ is  nc continuous if and only if  the 
  linear extension $\widetilde{g_1} : X_K \to X_L$ is continuous with respect to the canonical TVS topologies on $X_K$ and $X_L$; and  if and only if  $g_1$ is continuous. 
\end{prop}

\begin{proof}  The one direction of the first `if and only if' is obvious.
For the other direction,  Lemma \ref{a}  shows that $g_1$ extends to a 
linear map $\widetilde{g_1} : X_K  \to X_L$ 
If $k \in K_n$ then since $\xi^*  g_n(k) \xi = g_1(\xi^* k \xi)$ for unit vectors $\xi \in \bC^n$, we have by 
 an argument in the Remark after Proposition \ref{isaf} 
  that  $\widetilde{g_1}^{(n)}(k) = 
g_n(k)$.  That is,  $g_n$ may be identified with the restriction of $\widetilde{g_1}^{(n)}$ to the copy of $K_n$ in $X_K$.
Thus  $g$ is matrix affine.  
 
 If $\widetilde{g_1}$ is continuous then so is $\widetilde{g_1}^{(n)}$ for all $n$, and hence so is $g_n$, being a restriction of $\widetilde{g_1}^{(n)}$.  Conversely, if $g_1$ is continuous then so is $\widetilde{g_1}$
 by Lemma \ref{ab} (4). \end{proof}

\subsection{The universal operator space $X_K$ of a complex matrix convex set}  \label{uos} 

If $K$ is a selfadjoint complex matrix convex set in a complex $*$-vector space $V$, and if $K_1$ is contained in a hyperplane not passing through 0, let  
 $$C_n = \{ \gamma^* k \gamma : k \in K_m , \gamma  \in M_{m,n} : m \geq n \} \subset M_n(V)_{\rm sa}, \qquad 
\cC = (C_n).$$  This is the {\em matrix cone  generated by} $K$.   It is well defined independently of the 
superspace $V$ up to `nc affine zero-preserving isomorphism'.
Indeed any nc affine embedding $\epsilon : K \to W$  into another complex $*$-vector space extends to a linear isomorphism taking $\cC$ bijectively onto the matrix cone  generated by $\epsilon(K)$, by an argument 
in the Remark after Proposition \ref{isaf}.

We will review basic definitions of nc base norm spaces at the start of Section \ref{mcs}.  For now we merely state: 

\begin{thm} \label{xknp}   Suppose that $K = (K_n)$ is an abstract complex   matrix  compactly convex set.
Then $K$ may be viewed as a selfadjoint complex compact matrix convex  set in the complexification $Z$ of  the real base norm space $X_{K_1}$ constructed from $K_1$ in Section {\rm \ref{wiaccu}}. If $\cC$ is the 
matrix cone in $Z$ generated by this copy of $K$ then 
$X_K = (Z,\cC,K)$ is a complex nc base norm space with nc base $K$.  If in addition $K_1$ is locally convex then  
$X_K$ is the dual of an operator system,  and is a nc dual base norm space with nc dual base $K$.   In this case $X_K \cong \bA(K)^*$ as 
 nc dual base norm spaces (via a map which is a weak* homeomorphism, a nc base morphism, and a completely isometric complete order isomorphism).
\end{thm} 

We will prove this in Section \ref{mcs}, together with many results which are the variants for the nc/matrix setting of results from Sections above, or from Section \ref{abs} below.
For example, we will have the expected universal property of $X_K$: Any nc bounded (resp.\ positive) matrix affine map  $f : K \to V_{\rm sa}$  into a complex $*$-operator space $V$, has a unique complex  linear completely 
bounded (resp.\ completely positive) extension $\tilde{f} : X_K \to V$.
Indeed since by Theorem \ref{xknp} the universal operator spaces $X_K$ are just the nc base norm spaces with nc base $K$, their study 
may be viewed  in some sense as essentially a part of the theory of nc base norm spaces. 

\section{Applications to base norm spaces} \label{abs} 

\begin{cor} \label{breg} Suppose that $K$ is a compact convex set in a real (resp.\ complex) Hausdorff LCTVS  (resp.\ $*$-LCTVS) $(V,\tau)$,  with $K$ lying in a real hyperplane $H$ in $V$ (resp.\ $V_{\rm sa}$) not passing through 0, 
then $V$ contains (continuously) a dual base norm space $X$ with base $K$, and whose weak* topology agrees with $\tau$ on base-norm bounded sets in $X$.  \end{cor} 

\begin{proof}  
Let $X =  \bR_+ K -  \bR_+ K$ in $V$. Since $K$ is  compact, topologically convex, and  locally convex, 
$X_K$ is a dual base norm space by Lemma \ref{r} (which may be identified with $A(K)^*$).  Thus $X$ (resp.\ $X + iX$) is a base  norm space with base $K$, since $i_K : X_K \to X$ 
(resp.\ $i_K : (X_K)_c \to X + iX$)  is an isomorphism by Lemma \ref{a} (3), 
and $i_K$  is a homeomorphism between the copies of  $K$.   In the $*$-vector space case, $i_K$ is selfadjoint, and $X = i_K(X_K)_{\rm sa}$ is a real dual base norm space.  
On bounded sets of $X$  the weak* topology agrees with $\tau$ by the usual converging subnet argument, using that a net  converges to $x$ if and only if every subnet has a converging subnet with limit $x$. 
Indeed suppose that a bounded net $(c_t k_t  - d_t k_t')$ converges to $x$ in $X$ 
in one of these topologies.  Here $c_t, d_t \geq 0, k_t, k_t' \in K$.  
Consider a subnet, which for convenience we continue to write  as $(c_t k_t  - d_t k_t')$. We can 
assume that $(c_t)$ and $(d_t)$ are bounded nonnegative scalars, 
and have subnets with limits $c, d$. We can also assume that $k_{t_\mu} \to k, k'_{t_\mu} \to k'$ in the topology 
of $K$, hence in $\tau$ and in the weak* topology. 
Thus $c_{t_\mu} k_{t_\mu}  - d_{t_\mu} k_{t_\mu}' \to c k - d k'$ in both topologies.  So $x = c k - d k'$.  
Thus the isomorphism $i_K$ is bicontinuous, indeed is a homeomorphism, on base norm bounded sets.  
 (In the complex case $X + i X = i_K(X_K)$,  and the relationship between the weak* topology 
on the homeomorphic space $X_K$ to the weak* topology on its selfadjoint part is easy to understand 
\cite[Section 3]{BH}.) That $i_K$   is continuous as a map into $V$,  follows from Lemma \ref{ab} (4). 
 \end{proof} 
 
{\bf Remark.} The `continuously' in the statement of the last Corollorary is not necessarily `bicontinuously', as can be seen by considering a topology on $B(H)^*$ say which agrees on $K = {\rm Ball}(B(H)^*)$ with the weak* topology. 

\bigskip 

In some sense Corollary \ref{breg}  improves on, or refines, a `classical regularity result' from \cite{Breg}.   
It can be stated as saying the following: 

\begin{cor} \label{breg2}  If $(V,\tau)$ is a real (resp.\ complex) Hausdorff LCTVS  (resp.\ $*$-LCTVS), and if $K$ is a compact convex set 
which spans $V$ and lies in a hyperplane in $V$ (resp.\ $V_{\rm sa}$) not passing through 0 (that is, $K$ is preregularly embedded in $V$ in the language of {\rm \cite{Breg}}) 
then $V$ (with a possibly finer topology) is a dual base norm space with base $K$.  The  weak* topology of this dual base norm space agrees with 
$\tau$ on base-norm bounded sets.  Thus $A(K)^* \cong (V,\tau)$ via a continuous  isomorphism and base norm space isomorphism which 
is a homeomorphism on base-norm bounded sets.  \end{cor} 

\begin{proof} As in the proof of Corollary \ref{breg}  we get a continuous isomorphism $i_K  : X_K \to V$, which in this case is clearly surjective, and we get the other stated consequences.  The isomorphism with $A(K)^*$ follows from the theory of 
base norm spaces (see the lines above \cite[Theorem 4.13]{BH}, and the proof of that result).
 \end{proof}

\begin{cor} \label{breg25}  If $(V,\tau)$ is a real (resp.\ complex) ordered Hausdorff LCTVS  (resp.\ $*$-LCTVS) with generating cone $V_+$. 
Suppose that either  $V_+$ has a $\tau$-compact base $K$, or that $V_+$ is locally compact. 
Then $(V,V_+)$ (with a possibly finer topology) 
is a dual base norm space with compact base $K$ for $V_+$, and all the other conclusions
of Corollary {\rm  \ref{breg2}} hold.  \end{cor}

\begin{proof}  In the first case, the base function for $K$ extends to $V$ and defines a  hyperplane not passing through 0.    
So all of the conditions of 
Corollary  \ref{breg2} hold.  The second case follows from the first, since it is known that a  locally compact proper cone has a compact base (see  \cite[p.\ 87]{Phelps}).  
 \end{proof} 

{\bf Remarks.} 1)\ Thus one may in a sense  {\em define} a dual base norm space to be an 
LCTVS with a compact convex subset as in Corollary  \ref{breg2}, or as in Corollary  \ref{breg25}. 
Thus e.g.\ dual base norm spaces `are' just the ordered LCTVS's with a locally compact generating cone.

\smallskip 

2)\ In $\bR^n$ it is known that a closed cone is proper if and only if it has a base.  Indeed in  this case the  closed cone is obviously
locally compact so by the fact in the last proof it has a compact base.   Also it is well known that in this setting any base for the cone is compact. 

\smallskip 

3)\ Generalizing Remark 2), we saw in Theorem \ref{llm} that proper locally compact cones in an LCTVS,  are up to affine homeomorphism the same as the abstract  locally compact locally convex cones  (Here `locally compact cone' assumes that the
 cone operations are continuous.)   
 
  \begin{cor} \label{s}  Let $(X,\leq,K)$ be a real or complex base norm space such  that the base  $K$ is endowed with a Hausdorff topology with respect to which it is topologically convex,  compact and locally convex.    Then $(X,\leq,K)$ is a dual base norm space, and  with respect to this duality $K$ is a weak* compact convex set.   \end{cor}

\begin{proof}  By Lemma \ref{r}   $K$ is affinely homeomorphic to a compact convex set in a Hausdorff LCTVS.  Thus 
$A(K)^*$ is a base norm space $X$ with compact base $K$.  Thus by Lemma \ref{bne}, $X \cong A(K)^*$ is a dual space  and this 
isomorphism is isometric and as base norm spaces.  So $X$ is a dual base norm space and with respect to this duality $K$ is weak* compact.  \end{proof} 

\begin{prop} \label{u}  There exists a base norm space $X$ with convex base $K$ which is compact in a 
metrizable topology on $X$ which is coarser than the norm topology, but $X$ is not a dual base norm space and $K$ is not compact with respect to any  LCTVS topology on $X$.   \end{prop}

\begin{proof}  Let $K$ be Roberts' example of a compact convex set in a metrizable  TVS, with no extreme points.  By Corollary \ref{tc}  $K$ is affinely homeomorphic to the compact base of a locally convex cone $\cC$ in a Hausdorff TVS.
By Lemma \ref{p} the  Grothendieck group of $\cC$ is a real vector space $X = \cC - \cC$ with base $K$ for $V_+ = \cC$, and indeed   $X$  is a base norm space with norm closed base $K$, and with the canonical base function 
of $K$ in $X$.   If $X$ were a dual base norm space then $X \cong A(K)^*$ is a dual space, so that $K$ has extreme points by the Krein-Milman theorem.  If $K$ was compact with respect to an LCTVS topology on $X$
then by Corollary  \ref{s} $X$ is a dual space, again giving a contradiction. Alternatively,  if $K$ were locally convex or if the dual of  $X$ in a LCTVS topology
supplied enough separating functions in $A(K)$ then $A(K)^*$ is a base norm space $X$ with base $K$.  Thus $X \cong A(K)^*$ is a dual space by Lemma \ref{bne}, again giving a contradiction. 
\end{proof}

\begin{cor} \label{ws}   A compact topologically convex set is affinely homeomorphic to a weak* compact
convex set in a dual Banach space  if and only if it is locally convex.   \end{cor}

\begin{proof}  The one direction is clear since the weak* topology is locally convex.
For the other,   $K$ is affinely homeomorphic to a compact convex set in a Hausdorff LCTVS $E$ by Lemma \ref{r}, and $E$ may be taken to be 
$A(K)^*$ and $K$  weak* compact by e.g.\ the proof of Corollary \ref{s}.     \end{proof}

We give a nice application to the weak* topology on a real dual base norm space $E$.  The following result applies in particular to the selfadjoint part of the dual of a von Neumann algebra or the dual of an operator system $\cS$, or on $C(K)^*$ for compact $K$ (or more generally to the selfadjoint part of any complex dual base norm space).   
Namely, it shows how to retrieve the weak* topology on $E$ 
from the weak* topology on $E_+$, or from the weak* topology on $K$
(by employing the Lawson cone construction from $K$).  Of course once one has the  weak* topology on
the selfadjoint part of the dual of a complex 
 operator system $\cS$, one has it on all of $\cS^*$ (by considering the product topology).

\begin{cor} \label{w}  The weak* topology on a real dual base norm space $E$, and in particular the weak* topology 
on the selfadjoint part of the dual of an operator system, 
 is the quotient topology on $E$ induced by the subtraction map $E_+ \times E_+ \to E$.
Here $E_+$ is given the topology of the Lawson cone of the base $K$ (which coincides with the relative weak* topology). 
 \end{cor}

\begin{proof}  If $K$ is the state space of $E$, which is weak* compact,  then $E \cong A(K)^*$.    Let $\cC = E_+$ and $X = (\cC \times \cC)/\sim$ with the quotient topology induced 
 by the subtraction map $E_+ \times E_+ \to E$.  Then $X$ is a Hausdorff  LCTVS  in which $\cC$ is embedded affine topologically by  Theorem \ref{llm}. 
 Also $X$ is a base norm space by Lemma \ref{p}.  Thus $X$ is a dual  base norm space with compact base $K$ by Corollary \ref{s}. So $E$ and $X$  are the same, indeed both are $A(K)^*$ as LCTVS's, and isometrically (see Lemma \ref{bne}).     \end{proof}

\bigskip 

{\bf Remark.} The category of complex base norm spaces is equivalent to the category of real base norm spaces.
Indeed the first half of the proof of \cite[Lemma 3.3]{BH} shows that any complex base norm space $F$ is the dual Taylor complexification of its selfadjoint part $F_{\rm sa}$.  Moreover any complex base morphism between complex base norm spaces is selfadjoint, hence restricts to 
a real base morphism between the selfadjoint parts.  Conversely a real base morphism $F_{\rm sa} \to G_{\rm sa}$
extends to a complex base morphism $F \to G$ by Lemma  \ref{ab}.

\section{Matrix convex sets and nc bases} \label{mcs}

In \cite{BH} we introduced and studied nc base norm spaces and nc dual  base norm spaces.    We refer the reader to \cite{BH} for various definitions of 
a nc base norm space, and we shall meet more equivalent definitions below.  It is a matrix ordered $*$-vector space and operator space with closed matrix cones, and  the base now is a closed 
matrix convex set  $K \subset ({\rm Ball}(M_n(X))_+)$, and every $x \in M_n(X)_+$ is of form $ckc$ for $c \in (M_n)_+$ and $k \in K_n$. 
  For every $t > 1$ and $x \in {\rm Ball}(M_n(X))_{\rm sa}$ we may write 
 $x = c_1 x_1 c_1 - c_2 x_2 c_2$ for  positive matrices $c_i$, with $\| c_1^2 + c_2^2 \| \leq t$  and $x_i \in K_n$.  This expression is a noncommutative version of 
  the classical condition mentioned at the start of Section \ref{wiaccu}  that Ball$(X)  \subseteq t \, 
{\rm co} (K \cup (-K) )$.  It also essentially defines the nc base norm on  $M_n(X)_{\rm sa}$ in terms of an infimum of such $\| c_1^2 + c_2^2 \|$ (see equation (2) in \cite{BH}).  
We then get the norm of a general $x \in M_n(X)$ by the formula $\| x \| = \| \tilde{x} \|$ where $\tilde{x}$ is the selfadjoint matrix with rows $[0 \; x]$ and $[x^* \; 0]$. 
 The nc base  function is (the unique)  selfadjoint functional $f$ on $X$ that is $1$ on $K_1$; indeed   
 $K_n = \{  x \in  M_n(X)_+ : f^{(n)}(x) = I_n \}$ for each $n$.   A nc dual base norm space  is
a nc base norm space with an operator space predual such that the nc base $K$ is 
weak* closed (and hence weak* compact).  The dual operator space of an operator system is the generic nc dual  base norm space (see below).   Equivalently, the dual  nc base norm spaces are exactly (up to appropriate isomorphism) the spaces $\bA(K)^*$ for a compact convex set $K$.  The dual  base of $\bA(K)^*$ is $\delta(K)$, where $\delta : K \to \bA(K)^*$ is the canonical map. The generic nc base norm space ``is''  the predual of a dual operator  system, with the base 
corresponding to the normal matrix state space.   A nc base morphism is a completely positive linear map $u : X \to Y$ between nc base norm spaces mapping nc base into nc base.  Again this is equivalent 
(assuming $u$ completely positive) to $f_Y \circ u = f_X$, where $f_X$ and $f_Y$ are the nc 
base functions.   Indeed any map taking nc base into nc base is completely positive, selfadjoint, and completely contractive. 

To see that the dual of an operator system  $\cS$ is a dual nc base norm system with base the nc state space $K = ({\rm UCP}(\cS, M_n))$, first observe that $K$ is the generic 
compact matrix convex set, and it is weak* compact at each level \cite{WW,DK}.  In the notation above $X = \cS^*$, so that $M_n(X) = M_n(\cS^*) \cong CB(\cS,M_n)$, with $M_n(\cS^*)_+ \cong {\rm CP}(\cS,M_n)$, the completely positive maps.
Every such  completely positive map may be written as $c u(\cdot) c$ for $c \in (M_n)_+$ and matrix state $u \in {\rm UCP}(\cS, M_n)$.  The selfadjoint part of $M_n(\cS^*)$  correspond to
the selfadjoint completely bounded maps from $\cS$ into $M_n$, which equals  CP$(\cS,M_n) - {\rm CP}(\cS,M_n)$.  
Any completely contractive map from $\cS$ into $M_n$ may be written as a 
difference of two completely positive maps whose sum has norm $\leq 1$ (this follows e.g.\ from \cite[Theorem 8.5]{Pnbook}).
In this case the nc base  function is `evaluation at'  $1 \in \cS$. 

A base norm space or nc base norm space $X$ with compact base $K$ has a canonical  TVS structure.
Indeed $X$ is isomorphic to the universal space $X_{K_1}$ or its complexification $Z$, so $X$ inherits the TVS structure 
in e.g.\ Lemma \ref{p} and the remark after it (we put the `product topology' 
 on the complexification).  This observation also shows that 
our next result will immediately give a nc analogue of Lemmas \ref{ab} and \ref{a}, 
taking $X$ to be the universal operator space $X_K$ of $K$, which we saw in Theorem \ref{xknp}   is a nc base norm space with nc base $K$ if $K$ is matrix compact.

\begin{lemma} \label{bnmn}   Suppose that $(X,K)$ and $(Y,L)$ are real or complex nc base norm spaces. 
A   matrix affine $f : K \to L$ extends uniquely to a completely contractive complex linear selfadjoint and completely positive nc base morphism $X \to Y$.   More generally a  matrix affine (resp.\ affine and nc bounded, affine and nc positive, affine and continuous) $f : K \to V_{\rm sa}$ into a real or complex $*$-vector space 
(resp.\ $*$-operator space, matrix ordered space, $*$-TVS) $V$ extends uniquely to a   linear selfadjoint (resp.\ completely bounded with 
$\| \tilde{f} \|_{\rm cb} = \| f \|_\infty$,  completely positive, continuous with respect to the 
canonical TVS structure of $X$  above (assuming $K_1$ compact)) map $\tilde{f} : X  \to V$. 
 \end{lemma}

\begin{proof}   We just do the complex case, the real being similar.  By the classical case in Lemma \ref{ab} (2) there is a unique complex linear selfadjoint extension $\tilde{f} : X  \to V$, and it is one-to-one  if $f$ is  one-to-one and maps into a hyperplane not passing through 0. Suppose that $f$ is nc bounded by a constant $D$ (that is, uniformly bounded at all matrix levels). 
If $x = x^*$ is written as is usual in \cite{BH} as $c^{\frac{1}{2}} y c^{\frac{1}{2}} - d^{\frac{1}{2}} z d^{\frac{1}{2}}$ then 
$$\| \tilde{f}^{(n)}(x) \| = \| c^{\frac{1}{2}} f(y) c^{\frac{1}{2}} - d^{\frac{1}{2}} f(z) d^{\frac{1}{2}} \| 
= \| [ c^{\frac{1}{2}}  \; d^{\frac{1}{2}}  ] (f(y) \oplus (-f(z)))  [  c^{\frac{1}{2}}  \; d^{\frac{1}{2}} ]^\tran \| \le D \| c+d \|.$$ 
 by the considerations in the proofs of \cite[Proposition 3.2]{Breg}
 and \cite[Lemma 4.3]{BH}. 
 It follows that $\tilde{f}$ is `completely bounded on selfadjoint matrices', with constant $\leq D$.  Since $\tilde{f}$ is 
 selfadjoint we have for any $x \in M_n(X)$ that 
$$\| \tilde{f}^{(n)}(x) \| = \| \tilde{f}^{(2n)}(\tilde{x}) \|  \leq D \| \tilde{x} \| = D \| x \|.$$  So  $\tilde{f}$ is completely bounded
with $\| \tilde{f} \|_{\rm cb} \leq \| f \|_\infty$.   The converse inequality holds since $\tilde{f} = I_n$ on $K_n$, and
elements of $K_n$ have nc base norm 1.
If $V = Y$ then   $\tilde{f}$ is a nc base morphism as in 
the earlier real and complex variants.    
  We leave the rest to the reader (e.g.\ the `continuous TVS case' follows from Lemma \ref{ab}(4)).  \end{proof}

The following shows that a nc base norm space is completely determined by its nc base: 
 
\begin{lemma} \label{bnen}  Suppose that $(X_i,K_i)$ are  real or complex nc base norm spaces and that $K_1$ is matrix affinely isomorphic 
to $K_2$.  
Then 
$X_1 \cong X_2$ completely  isometrically as base norm spaces.  Moreover this isomorphism is also a 
homeomorphism (resp.\ weak* homeomorphism)  if $K_i$ are matrix compact 
(resp.\ if $(X_i,K_i)$ are nc dual base norm spaces). \end{lemma} 

\begin{proof}   We just do the complex case, the real being similar.   The first assertions and the first homeomorphism result follow from Lemma \ref{bnmn}.   In the dual base norm case, if $K_1$ is matrix affinely homeomorphic to $K_2$ then by the duality of compact matrix convex sets and  operator  
systems \cite{WW}, 
$\bA(K_1) \cong \bA(K_2)$ as operator systems.  Since $X_i = \bA(K_i)^*$, we have $X_1 \cong X_2$ completely isometrically as dual base norm spaces. \end{proof}

\begin{theorem} \label{isnbns} Suppose that $K = (K_n)$ is a (complex) matrix convex set in a complex $*$-vector space $V$, with $K_n  \subset M_n(V)_{\rm sa}$ for each $n$. 
Suppose also that $K$ has a topology with respect to which it is an abstract matrix compactly  convex set. Define $$C_n = \{ \gamma^* k \gamma : k \in K_m , \gamma  \in M_{m,n} : m \geq n \} \subset M_n(V)_{\rm sa},$$ 
set $\cC = (C_n)$, 
and let $X$ be the complex 
span of $K_1$ in $V$.   Assume further that $(X_{\rm sa}, C_1)$ is a base norm space with base $K_1$, 
or equivalently 
(by  the lines after Lemma {\rm  \ref{ab}}), 
merely that $(X_{\rm sa},C_1)$ has linear base $K_1$. 
Then $(X,\cC,K)$ is a  nc base norm space whose canonical matrix norms make $X$ an operator space and a nc base norm space with nc base $K$. 

If in addition $K_1$ is locally convex then  
$X$ is a dual operator space and a nc dual base norm space with nc dual base $K$. 
 \end{theorem}

\begin{proof}  The matrix cone $\cC = (C_n)$ generated by $K$ is a matrix convex set with respect to $V$, and 
induces a matrix ordering for $V$  in the sense of operator system theory \cite{CE}.    Thus for example $C_n \oplus C_m \subset C_{n+m}$ 
and $\gamma^* C_m \gamma \subseteq C_{n}$ for $ \gamma  \in M_{m,n}$.   Let $X$ be the complex span of $K_1$.
Then $X_{\rm sa}$ is the real span of $K_1$, as may be seen by looking at $\frac{1}{2}(x + x^*)$ for $x \in X$.  
If $X_{\rm sa}$ has linear base $K_1$ then  
$(X_{\rm sa}, C_1)$ is a base norm space with base $K_1$ (for example by  the lines after Lemma \ref{ab}), 
and 
base function $f_1$.  
Then $X$ is a base norm space
in the sense of \cite[Section 3]{BH}, and is a complex base norm space with a canonical norm.  Define the base function $f = (f_1)_c$ on  $X$ by complexification.   
We have  $\xi^* f^{(n)}(k) \xi = f(\xi^* k \xi) = 1$ for all $k \in K_n$ and all unit vectors $\xi \in \bC^n$.
 Thus  $f^{(n)}(k) = I_n$.    Thus  
 $f^{(n)}( \gamma^* k \gamma ) = \gamma^*  \gamma$ for $k \in K_m , \gamma  \in M_{m,n}$.  Thus $f$ is completely positive on $(X,\cC)$. 
  If $\gamma_1^* x \gamma_1 = -\gamma_2^* y \gamma_2$, then 
$\gamma_1^* x \gamma_1 + \gamma_2^* y \gamma_2 = 0$.   Applying the base function we see that
$\gamma_1^* \gamma_1 + \gamma_2^* \gamma_2 = 0$, and so  $\gamma_1 =  \gamma_2 = 0$.
Thus $\cC$ is  a proper matrix cone/ordering for $V$, and 
 $(V, \cC)$ is a (proper) matrix ordered $*$-vector space.  Also,   $K$ is in a nc hyperplane $H = (H_n)$ not passing through 0,
 the one defined by $f$.  Indeed if 
 $$f^{(n)}( \gamma^* k \gamma ) = \gamma^*  \gamma = I, \qquad k \in K_m , \gamma  \in M_{m,n},$$ 
 then $\gamma^* k \gamma \in K_n$. Thus $K_n = M_n(X)_+ \cap H_n$.
 Thus  Condition (a) holds in the definition of a matrix base norm space above \cite[Theorem 4.4]{BH}.   Condition (b) in that definition
then   follows from \cite[Lemma 4.2 (1)]{BH}.   Indeed from that Lemma we see that we can take $m = n$ in the definition of $C_n$ above.

Now assume in addition that $K_n$ is compact and matrix topologically convex.   We verify Condition (c) holds in the definition of a matrix base norm space above \cite[Theorem 4.4]{BH}.  So suppose that $x = x^* = y-z \in M_n(X)_{\rm sa}$, where $y,z \in C_n$ are fixed.
Suppose that $x = c_n^{\frac{1}{2}}  k_n c_n^{\frac{1}{2}} - d_n^{\frac{1}{2}}  k'_n d_n^{\frac{1}{2}}$ for $c_n, d_n \in M_n^+, k, k' \in K_n$, with $\| c_n \| + \| d_n \| \leq 1/n$. 
Then $y +  d_n^{\frac{1}{2}}  k'_n d_n^{\frac{1}{2}} = z + c_n^{\frac{1}{2}}  k_n c_n^{\frac{1}{2}}$.  Applying the base function we see that
$f^{(n)}(y) + d_n = f^{(n)}(z) + c_n$.   Call this $b_n$.  By multiplying $x$ by a positive scalar we may assume that  $0 \leq b_n \leq I_n$, and set $a_n = I - b_n$.  For some fixed $w \in K_n$ we have 
 $$y +  d_n^{\frac{1}{2}}  k'_n d_n^{\frac{1}{2}} + a _n^{\frac{1}{2}}  w a_n^{\frac{1}{2}} = z + c_n^{\frac{1}{2}}  k_n c_n^{\frac{1}{2}} + a _n^{\frac{1}{2}}  w a_n^{\frac{1}{2}} ,$$
 and this is a matrix convex combination, so lives in $K_n$.  We may replace by a subsequence so that $a_n \to a$ say.  Letting $n \to \infty$ and using that $K$ is topologically convex,   we deduce that $y + a^{\frac{1}{2}}  w a^{\frac{1}{2}} = z + a ^{\frac{1}{2}}  w a^{\frac{1}{2}},$ and so $y = z$.  Hence $x = 0$ as desired. 
  
  Thus $(X,\cC,K)$ is a  matrix base ordered  space.  Moreover by \cite[Theorem 4.4]{BH}, $M_n(X)$ has a natural norm for all $n$, with respect to which 
  $X$ is an operator space.  Indeed we have checked that $X$ satisfies all of the conditions to be a nc base norm space with nc base $K$ with the  exception  of $M_n(X)_+$ (and therefore also $K_n$) being closed
   in the norm topology for $n \in \bN$.  To prove this we use the idea of the argument in the last paragraph. Suppose that $x_n = r_n^{\frac{1}{2}}  k_n r_n^{\frac{1}{2}} \to x = y-z$ in this norm, with $y, z \in M_n(X)_+, k \in K_n, r_n \in (M_n)_+$.   
  Applying the base functional we see that $r_n \to f^{(n)}(x)$, so that $(r_n)$ is convergent to $c$ say, and is bounded.  
  Since $\| x_n - y + z \| \to 0$ we may write $$x_n - y + z = c_n^{\frac{1}{2}}  y_n c_n^{\frac{1}{2}}  - 
  d_n^{\frac{1}{2}} z_n d_n^{\frac{1}{2}},$$
  with
$0 \leq c_n, d_n \to 0$, and $y_n, z_n \in K_n$.  Thus
$$x_n + z +
  d_n^{\frac{1}{2}} z_n d_n^{\frac{1}{2}} = y + c_n^{\frac{1}{2}}  y_n c_n^{\frac{1}{2}} .$$
     Applying the base function shows that $r_n + f^{(n)}(z) + d_n = f^{(n)}(y) + c_n.$ Call this $b_n$.
     By scaling we may assume that  $0 \leq b_n \leq I_n$, and set $a_n = I - b_n$.  For some fixed $w \in K_n$ we have 
   $$x_n + z +
  d_n^{\frac{1}{2}} z_n d_n^{\frac{1}{2}} + a _n^{\frac{1}{2}}  w a_n^{\frac{1}{2}} = y + c_n^{\frac{1}{2}}  y_n c_n^{\frac{1}{2}} 
  + a _n^{\frac{1}{2}}  w a_n^{\frac{1}{2}}.$$ This is a matrix convex combination in $K_n$.   
   Suppose that a subnet  $k_{n_t} \to k \in K_n$ in the topology there.   As in the last paragraph,
  by topological convexity 
  we deduce that $$c^{\frac{1}{2}}  k c^{\frac{1}{2}} + z + a^{\frac{1}{2}}  w a^{\frac{1}{2}}= y + a ^{\frac{1}{2}}  w a^{\frac{1}{2}}.$$ 
  Thus $x = y - z = c^{\frac{1}{2}}  k c^{\frac{1}{2}} \in C_n$.   So $C_n$ and $K_n$ are closed in the nc base norm, and moreover the topology on $K_n$ is coarser than the base norm topology.    

If in addition $K_1$ is locally convex then $(X_{\rm sa},K_1)$ is a dual base norm space by Corollary \ref{s}, and  with respect to this duality $K_1$ is a weak* compact convex set.   The weak* topology on $X_{\rm sa}$  is coarser than the norm topology.  We extend this topology to $X$, and to $M_n(X)$ for $n \in \bN$, in the obvious way (the `product topology').   This topology, which we write as $\tau$, is coarser than the 
norm topology on $X$, since $X$ is a simple norm complexification of $X_{\rm sa}$, so that we are just dealing with product topologies. 

We show that Ball$(M_n(X))$ is compact in this topology.   
Indeed if $(x_\lambda)$ is a net in Ball$(M_n(X))$, then 
$(\widetilde{x_\lambda})$ is a net in Ball$(M_{2n}(X))$.
Thus we may write $\widetilde{x_\lambda} = c_\lambda^{\frac{1}{2}} k_\lambda c_\lambda^{\frac{1}{2}} - d_\lambda^{\frac{1}{2}} k'_\lambda d_\lambda^{\frac{1}{2}}$.   We may assume  by taking a subnet 
that $\| c_\lambda + d_\lambda \| \leq t_\lambda \to 1$.   (We are using here a little-met trick with nets, that a net 
 $(x_t)$  of elements in the unit ball of a normed space each of whose norms is an infima of numbers
 $s^n_t$ for $n \in \bN$ with $s^n_t$ depending on  representatives $y^n_t$ of $x_t$ for $n \in \bN$, so that $x_t = y^n_t$, 
has a subnet whose elements are   representatives of the original 
$x_t$ with norm dominated by numbers converging to 1.  To see this,
consider the doubly indexed net $(y^n_t)$ as a subnet of $(x_t)$, and note that this subnet has norm dominated by numbers converging (with respect to the product directed set) to $1$.)

We may assume that subnets of all of the matrices converge.  Replacing the subnets by nets,  $c_\lambda \to c$ and $d_\lambda \to d$ in (matrix) norm, and $k_\lambda \to k, k'_\lambda \to k'$ in the topology on $K_{n}$.  This implies as before (using (\ref{pn})) that the entries in $k_\lambda$ converge in $\tau$ to the entries in $k$. 
We also have that the entries of 
$\widetilde{x_\lambda}$ converge in $\tau$ to elements of $X$, thus $\widetilde{x_\lambda}$ converges in $\tau$  to a matrix $z$ in $M_{2n}(X)$. This matrix may be written (by continuity of the matrix entries) as $c^{\frac{1}{2}} k c^{\frac{1}{2}} - d^{\frac{1}{2}} k' d^{\frac{1}{2}}$, and here $\| c + d \| \leq 1$.  Thus $z$ has norm $\leq 1$.  Therefore $x_\lambda$ converges in $\tau$ to a corner of $z$ which has norm $\leq 1$.   Thus Ball$(M_n(X))$ is compact for each $n$.
By the Dixmier-Ng theorem $X$ is a dual Banach space.  Hence  $X$ is a dual operator space
by  1.6.4 in \cite{BLM}.  The base function $f$ is weak* continuous on $X_{\rm sa}$, and so therefore also on 
$X$.   As stated in \cite{BH}, this implies that $X$ is a nc dual base norm space with nc dual base $K$.  Indeed we have 
essentially checked the conditions in the definition of the latter (see \cite[Definition 4.6]{BH}). 
\end{proof} 

{\bf Remark.} We continue the discussion in the last lines of the proof concerning the fact stated in \cite{BH} that 
nc dual base norm spaces may be characterized as a nc base norm space which is an operator space dual with the nc base weak* compact.  Indeed this fact follows easily from the last assertion of Theorem \ref{isnbns}, or from the last part of its proof. 

\begin{cor} \label{xkn}   Suppose that $K = (K_n)$ is an abstract complex matrix compactly convex set. 
Then $K$ may be viewed as a complex compact matrix convex set in the complexification $Z$ of  the real base norm space constructed from $K_1$ in Section {\rm \ref{wiaccu}}. If $\cC$ is the matrix cone in $Z$ generated by this copy of $K$ then 
$X_K = (Z,\cC,K)$ is a nc base norm space with nc base $K$. 
If in addition $K_1$ is locally convex then  
$X_K$ is a dual operator space and a nc dual base norm space with nc dual base $K$.  
\end{cor} 

\begin{proof}  If $\cC$ is the matrix cone  generated by $K$ then 
by Theorem \ref{isnbns} we have that  $(Z,\cC,K)$ is a nc base norm space $X_K$ whose canonical matrix norms make $Z$ an operator space and a nc base norm space with nc base $K$. 
If in addition $K_1$ is locally convex then  
$X_K$ is a dual operator space and a nc dual base norm space with nc dual base $K$.  
\end{proof} 

If  $K$  is an abstract real or complex matrix or nc convex set we define the nc bounded affine functions 
$\bA_{\rm b}(K)$ just as in \cite{DK,BMcI}.   As in those sources, just as in the $\bA(K)$ case  it is an  operator system
by the abstract characterization of 
operator systems.   If  $f : K \to M_n$ is nc affine and selfadjoint 
then since $-cI \leq f(k) \leq cI$ if and only if $\| f(k) \| \leq c$, for $k \in K$, it follows that 
the matrix norms on $\bA_{\rm b}(K)$  coincide with the matrix order unit norm on selfadjoint elements.
Hence by the usual trick they coincide on all elements.

\begin{lemma}  \label{xknco}  If $X$ is a complex nc base norm space with nc base $K$ then 
$\bA_{\rm b}(K) \cong X^*$  completely isometrically and complete order isomorphically via a map taking the identity $1$ to the base function for $K$.     In particular,  $\bA_{\rm b}(K)$ is a dual operator system in the sense of 
{\rm \cite{BMag}}. 
If further $K$  is an abstract complex matrix compactly convex set then $\bA_{\rm b}(K) \cong X_K^*$ completely isometrically and complete order isomorphically. \end{lemma}

\begin{proof}  The canonical (restriction to $K$) map $\rho : X^* = X_K^* \to A_{\rm b}(K)$ is clearly a completely positive 
complete 
contraction, and takes the base function for $K$ to the identity $1$.  It is 
surjective by Lemma \ref{bnmn}, and one-to-one by the uniqueness there.   Suppose that 
$f \in {\rm Ball}(M_n(A_{\rm b}(K)))_{\rm sa}$.
Hence   $f : K \to (M_n)_{\rm sa}$, and  $\| f \|_\infty \leq 1$.  By Lemma \ref{bnmn}
we have a completely contractive selfadjoint  extension $\tilde{f} : X \to M_n$,
and moreover if $f \geq 0$ then $\tilde{f}  \geq 0$.    Clearly $\rho(\tilde{f}) = f$.
For general $f \in {\rm Ball}(M_n(A_{\rm b}(K)))$ we have that $g \in {\rm Ball}(M_{2n}(A_{\rm b}(K)))_{\rm sa}$,
 where $g$ is the usual matrix with off-diagonal entries $f$ and $f^*$.  This extends to a completely contractive selfadjoint  
 $\tilde{g} : X \to M_{\rm 2n}$.  If $\tilde{f}$ is the `$1$-$2$-corner' of $\tilde{g}$ as usual, then 
 $\tilde{f}$ is a completely contractive   extension of $f$:  $\rho(\tilde{f}) = f$.  Thus $\rho$ is a 
  complete isometry and complete order isomorphism.  
\end{proof} 

{\bf Remark.} The last result generalizes the fact from \cite{DK} that $\bA_{\rm b}(K)$ is a dual operator system.
Indeed in  that setting, namely in the 
case that $K$ is an abstract complex matrix compactly convex set with $K_1$ locally convex, 
it is actually a bidual operator system.  That is for such $K$ we have $\bA_{\rm b}(K) \cong \bA(K)^{**}$
completely isometrically and as operator systems. 

\begin{cor} \label{tdosy}   The dual operator systems are precisely, up to unital complete order isomorphism
which is a weak* homeomorphism, the spaces $\bA_{\rm b}(K)$ for 
a matrix convex set $K$.  \end{cor}

\begin{proof}    This follows from the last result, since it is shown in \cite[Lemma 3.3]{BH} that the dual operator systems are precisely the duals of nc base norm spaces. 
\end{proof}

\begin{cor} \label{wsn}   A  complex compact matrix  convex set $K$ is  matrix affinely homeomorphic to a weak* compact
matrix convex set in a dual operator space if and only if $K_1$ is locally convex.  In this case 
  $\bA(K)$ `completely separates points' of $K$.  \end{cor}

\begin{proof}  The one direction is clear since the weak* topology is locally convex.
For the other, by Corollary \ref{xkn} we see that  $K$ is matrix affinely homeomorphic to a compact matrix convex set in a Hausdorff LCTVS $E$.
Indeed $E$ may be taken to be $\bA(K)^*$ by the usual categorical duality \cite{WW}, and $K$ is weak* compact there.
It is well known that $\bA(K)$ `completely separates points' of a complex compact matrix  convex set $K$ in a 
dual operator space. (See e.g.\ \cite{WW,DK}; indeed an operator system obviously `completely separates points' of 
its noncommutative state space.)
 \end{proof} 

 \begin{cor} \label{vn}   Suppose $(X,\leq,K)$ is a complex nc base norm space with $K$ matrix compact with respect to an  LCTVS topology on $X$,
 or merely with $K$ an abstract matrix compactly convex set with $K_1$ locally convex.  
Then $(X,\leq,K)$ is a dual nc base norm space, and with respect to this duality $K$ is a weak* compact matrix convex set. \end{cor}

\begin{proof}  
By the real case in the last Section (see e.g.\ Corollary \ref{s}),   
$(X_{\rm sa},\leq,K_1)$ is a real dual base norm space, and $K_1$ is weak* compact and locally convex in $X_{\rm sa}$.  
Hence by Theorem \ref{isnbns}  we have that $X$ is a nc dual base norm space with nc dual base $K$, indeed
$X \cong \bA(K)^*$ as  nc dual base norm spaces by \cite{BH}.  Hence $X$ is a
 dual operator space with respect to whose duality $K$ is  weak* compact. \end{proof}

\begin{cor} \label{bregcn} Suppose that $K$ is a selfadjoint compact  matrix  convex set in a complex   $*$-LCTVS  $(V,\tau)$,  and suppose that 
$K_1$ lies in a hyperplane not passing through 0. Then $V$ contains (continuously) a complex matrix dual base norm space $X$ with nc base $K$, and whose weak* topology agrees with $\tau$ on base-norm bounded sets in $X$. 
 \end{cor} 

\begin{proof}   By e.g.\ Corollary \ref{breg}, $V_{\rm sa}$ contains (continuously) a real dual base norm space with  base $K_1$, and whose weak* topology agrees with $\tau$ on base-norm bounded sets.  This subspace of $V$ 
is the real span of $K_1$, and is $X_{\rm sa}$ if $X$ is the complex span of $K_1$.  Theorem \ref{isnbns} implies that with the canonical cones $\cC = (C_n)$, 
 $(X,\cC,K)$ is a  matrix dual base norm space whose canonical matrix norms make $X$ an operator space and a nc base norm space with nc dual base $K$.   The associated weak* topology on
 $X$ is, by the proof of Theorem \ref{isnbns}, the product topology or complexification of
 the weak* topology in the first lines of the present proof, and so agrees with $\tau$ on bounded sets in $X$.
 We are also using the fact that $\tau$ is the product topology induced by its restriction to $V_{\rm sa}$, which 
 is obvious in a  $*$-LCTVS. \end{proof} 
 
 {\bf Remark.}  Note that a complex matrix dual base norm space $X$ with nc base $K$ is also (with an equivalent norm)
 a complex dual base norm space with base $K_1$.  It is easy to see by an argument similar to one in the last proof, that the two associated weak* topologies are the same. 
 
 \bigskip

Again in some sense Corollary \ref{bregcn} allows a refinement on the main regularity results of \cite{Breg}.   Just as in the `commutative case' in Corollary \ref{breg2}, this refinement of regularity can be stated as saying the following: 
\begin{cor} \label{breg2cn}  If $(V,\tau)$ is a complex  $*$-LCTVS, and if $K$ is a selfadjoint compact matrix convex set 
such that $K_1$  spans $V_{\rm sa}$ and lies in a hyperplane not passing through 0, 
then $V$ (with a possibly finer topology) is a complex nc dual base norm space with nc base $K$, and whose weak* topology agrees with 
$\tau$ on base-norm bounded sets.  Thus $\bA(K,\bC)^* \cong (V,\tau)$ via a continuous   selfadjoint  isomorphism and complex nc base  isomorphism which
is a homeomorphism on base-norm bounded sets.  \end{cor}

 \begin{thm} \label{breg25n}  Suppose that  $(V,\tau)$ is a complex matrix ordered $*$-LCTVS with closed selfadjoint  cones $M_n(V)_+$  for $n \in \bN$.
 Suppose also that either  {\rm (a)}\  
 $V_+$  is locally compact and spans $V$, or  {\rm  (b)}\ $V$ (at level 1) has a linear base which is $\tau$-compact.
Then $V$   is a dual nc base norm space with a compact matrix base for $(M_n(V)_+)$, and all the other conclusions
of Corollary {\rm  \ref{breg2cn}} hold.  \end{thm}

\begin{proof}  In case (b), as in Corollary \ref{breg25} we see that $V$ is a dual base norm space with compact base $K_1$ say.  It has a weak* 
continuous base function $f$ which defines a hyperplane not passing through 0, containing $K_1$.  The weak* topology agrees with $\tau$ on base-norm bounded sets. 
Set $$K_n = \{ [x_{ij}] \in M_n(V)_+ : f^{(n)}(x) = I_n \}.$$   Then $K = (K_n)$ is easily checked to be
matrix convex, and $K_n$ is selfadjoint since $f$ is selfadjoint on $V$. 
   To see that $K_n$ is compact in the product topology on 
$M_n(V)$, consider a net 
$x^t \in K_n$.   For fixed $i$ we have that $(x_{ii}^t)$ is in $K_1$, and so it 
has a convergent subnet with limit $z_{ii}$ say.  Moreover  $f(x_{ii}^t) = 1$ so that $f(z_{ii}) = 1$. 
Similarly  if $i \neq j$ then by the polarization identity formula
in  (\ref{pn}) we see that $(x_{ij}^t)$ is base-norm bounded in $V$.  So it has a 
weak* convergent (hence $\tau$-convergent) subnet with limit  $z_{ij}$ say.  
Moreover  $f(x_{ij}^t) = 0$ so that $f(z_{ij}) = 0$.    Also $x^t \to [z_{ij}] \in K_n$.   Thus $K_n$ is compact. So all of the conditions of 
Corollary  \ref{breg2cn} hold. 

In  case (a) we see similarly by Corollary  \ref{breg25} that $V$ is a dual base norm space with compact base $K_1$ say, and weak* 
continuous base function $f$.  Now proceed as in case (b). 
 \end{proof}

Thus one can define a complex nc dual base norm space e.g.\ to be a 
complex matrix ordered  $*$-LCTVS $(V,\tau)$
such that $V$ (at level 1) has a $\tau$-compact base, as in Theorem  \ref{breg25n}.   
 As we saw above $\tau$ may not be the desired weak* topology,
but it may be switched with it, and in any case is the same on `bounded sets'. 

\medskip

{\bf Example.}   The following example is often useful in constructing potential counterexamples to proposed variants of results above, since it can allow one to embed
a badly behaved operator space $X$ in a nc base normed space with a nc base that is usually tractable.

Let $X$ be an operator space (resp.\ let $W$ be a dual operator space), a  subspace (resp.\ weak $*$-closed subspace) 
of $B(H)$ say.  Then its Paulsen system (see e.g.\ early pages in Section 8 of \cite{Pnbook}) is
a   weak $*$-closed operator system in  $M = M_2(B(H)) \cong B(H^{(2)})$.
In  \cite[Section 5]{BH} we defined the {\em base Paulsen system} ${\mathfrak S}(X)$, which as a matrix ordered space is the same as $\cS(X)$, but is 
considered as a nc base norm space with base function the `normalized trace' $\tau$ on $\cS(X)$, namely the sum of the two scalars on  the 
main diagonal of $\cS(X)$.   This is strictly positive on  $\cS(X)_+$ so that $K_1 = \{ a \in \cS(X)_+ : \tau(a) = 1 \}$  is a base for 
$\cS(X)_+$, indeed is a linear base  for $\cS(X)$ in the sense of the introduction. 
Also in the dual operator space case $K_1$ is weak* compact, since it is weak* closed and bounded in  $M$.  Thus ${\mathfrak S}(X)$ is a 
dual nc base norm space by Theorem  \ref{breg25n}.

We showed in   \cite[Section 5]{BH} that the nc base norms on the base Paulsen system ${\mathfrak S}(X)$
 are equivalent (up to a constant) to the operator system norms. 
With some effort one can show the interesting fact that, just as is the case with the Paulsen system, the canonical map of an operator space $X$ into 
the $1$-$2$ corner of the base Paulsen system ${\mathfrak S}(X)$ is an  isometry.
So too are the canonical inclusions of the other three corners 
into ${\mathfrak S}(X)$.  (Note that there is a multiplication by $2$ in these inclusions because we are using the `normalized trace' 
instead of the trace', 
indeed obviously the norm of $E_{11} I$ in ${\mathfrak S}(X)$ is 1/2.)  Thus every operator space $X$ `is a corner' of a nc base norm space, its base Paulsen system. 

Similarly the projections onto the spaces sitting in the
four corners of ${\mathfrak S}(X)$ are complete contractions on ${\mathfrak S}(X)$ (multiplied by 2).   Also, the variant of `the Paulsen lemma' for the base Paulsen system holds.  We will present this elsewhere. 

\medskip 

We end with some remarks on functorial properties of base morphisms.  A one-to-one  base morphism $f : X \to Y$ between base norm spaces need not be an isometry, however its range is a base norm space and $f$ is an isometry (and order embedding) into that space, by Lemma \ref{bne}.  
Similarly in the nc case, using Lemma \ref{bnen}.   Quotients of base norm spaces behave well:

\begin{lemma} \label{qb}  A base morphism $u : X \to Y$ between real or complex base norm spaces with 
$u(K_X) = K_Y$ (these are the bases) is a quotient map, and induces
an isometric isomorphism $X/{\rm Ker}(u) \cong Y$.  Thus $X/{\rm Ker}(u)$ is a base norm space.  
Similarly a nc base morphism $u : X \to Y$ between nc base norm spaces with 
$u(K_X) = K_Y$ is a complete quotient map: it induces a completely isometric isomorphism $X/{\rm Ker}(u) \cong Y$. 
\end{lemma} 

\begin{proof} 
To see the first assertion in the real case notice that $cx - d y$ in $Y$ with $c + d < t$ and $c, d \geq 0, x, y \in K_Y$ 
may be lifted to a $ck - d k' \in X,$ with $k, k' \in K_X$.  The complex case follows from the real case.  Indeed  if $z$ is in the dual Taylor complexification of $Y$, with $\| z \| < 1$, then by the first lines in \cite[Section 3]{BH}
 we can write $z = \sum_{j=1}^n \, \alpha_j \, y_j$ with $y_j \in K_Y, 
\alpha_j \in \bC$, with $\sum_{j=1}^n \, | \alpha_j | < 1$.  If $u(k_j) = y_j$ and $w =  \sum_{j=1}^n \, \alpha_j \, k_j$,
then $u(w) = z$ and $\| w \| < 1$ as desired.  

To prove the nc case  suppose that $z \in M_n(Y)$.  
If $z$ is selfadjoint then an obvious variant of the argument above in the classical real case shows that 
$z$ lifts to a matrix in $X$ with close norm.  If $z$ is nonselfadjoint and 
$\| z \|_n < 1$, then $\| \tilde{z} \| < 1$.   Then if $\tilde{z} = cxc - d yd$ in $M_{2n}(Y)$ with $c^2 + d^2 \leq  t I$
then $\tilde{z}$ may be lifted to $w = ckc - d k' d \in M_{2n}(X)$.  
 Let $x$ be the $1$-$2$ corner of $w$.
Then $u_n(x) = z$ and $\| x \| = \| \tilde{x} \| \leq \| w \| \leq t$. \end{proof}

\section{Closing remarks and Acknowledgements.}

The real case of the nc theory in the last section is much more technical, and requires the complexification theory from 
\cite{BMcI} (or its matrix convexity variant).  Thus we defer this to a sequel, together with the  case of 
our results relevant to Davidson and Kennedy's nc convexity \cite{DK}.   
\bigskip

We thank D. M. Hay for some comments on a preliminary version of this paper, and spotting several typos.  
This project was partially supported by NSF grant DMS-2154903.

\end{document}